\documentclass[a4paper, 12pt]{amsart}
\usepackage{amsmath,latexsym,amssymb, hyperref}

\newif\ifpdf
\ifx\pdfoutput\undefined
\pdffalse % we are not running PDFLaTeX
\else
\pdfoutput=1 % we are running PDFLaTeX
\pdftrue
\fi

\ifpdf
\usepackage[pdftex]{graphicx}
\else
\usepackage{graphicx}
\fi

\newtheorem{lemma}{Lemma}[section]
\newtheorem{theorem}[lemma]{Theorem}
\newtheorem{corollary}[lemma]{Corollary}
\newtheorem{definition}[lemma]{Definition}

\newtheorem{remark}[lemma]{Remark}
\newtheorem{remarks}[lemma]{Remarks}
\newtheorem{proposition}[lemma]{Proposition}
\newtheorem{assumption}[lemma]{Assumption}

\newtheorem{example}[lemma]{Example}

\newtheorem{theorem*}[]{Theorem}

\def\N{\mathbb{N}}
\def\C{\mathcal{C}}

\def\R{\mathbb{R}}

\def\Aut{{\rm Aut}}

\def\X{\mathcal{X}}
\def\loc{{\rm orb}}
\def\orb{{\rm orb}}
\def\dom{{\rm dom}}

\def\b{\bar{b}}
\def\cl{{\rm cl}}
\def\h{\tilde{h}}
\def\g{\tilde{g}}

\def\S{\mathcal{S}}
\def\F{\mathcal{F}}
\def\id{{\rm id}}
\def\m{{\rm mult}}
\def\dim{{\rm dim}}

\def\K{{\mathcal{K}}}
\def\acl{{\rm acl}}

\def\Sym{{\rm Sym}}
\def\GL{{\rm GL}}
\def\Z{\mathbb{Z}}

\providecommand{\floor}[1]{\lfloor#1\rfloor}

\def\Ind#1#2{#1\setbox0=\hbox{$#1x$}\kern\wd0\hbox to 0pt{\hss$#1\mid$\hss}
\lower.9\ht0\hbox to 0pt{\hss$#1\smile$\hss}\kern\wd0}
\def\Notind#1#2{#1\setbox0=\hbox{$#1x$}\kern\wd0\hbox to 0pt{\mathchardef
\nn="3236\hss$#1\nn$\kern1.4\wd0\hss}\hbox to 0pt{\hss$#1\mid$\hss}\lower.9\ht0
\hbox to 0pt{\hss$#1\smile$\hss}\kern\wd0}
\def\ind{\mathop{\mathpalette\Ind{}}}
\def\nind{\mathop{\mathpalette\Notind{}}}

\usepackage{setspace}

\begin{document}

\title[]{Simplicity of the automorphism groups of some Hrushovski constructions}

\authors{

\author{David M. Evans}

\address{%
Department of Mathematics\\
Imperial College London\\
London SW7~2AZ\\
UK.}

\email{david.evans@imperial.ac.uk}

%\address{%
%School of Mathematics\\
%University of East~Anglia\\
%Norwich NR4~7TJ\\
%UK.}
%
%\email{d.evans@uea.ac.uk}

\author{Zaniar Ghadernezhad}

\address{%
School of Mathematics\\
 Institute for Research in Fundamental Sciences (IPM)\\
 P.O. Box 19395-5746 Tehran\\
 Iran.}

\thanks{The second author was supported by funding from the European Community's Seventh Framework Programme FP7/2007-2013 under grant agreement  23838.}

\email{zaniar.gh@gmail.com}

\author{Katrin Tent}

\address{%
Mathematisches Institut\\
Universit\"at M\"unster\\
Einsteinstrasse 62\\
48149 M\"unster, Germany.}

\email{tent@math.uni-muenster.de}

}

\date{}

\begin{abstract}  we show that the automorphism groups of certain countable structures obtained using the Hrushovski amalgamation method are simple groups. The structures we consider are the `uncollapsed' structures of infinite Morley rank obtained by the ab initio construction and the (unstable) $\aleph_0$-categorical pseudoplanes. The simplicity of the automorphism groups of these follows from results which generalize work of Lascar and of Tent and Ziegler. 
\end{abstract}

\maketitle

\section{Introduction} 

In this paper, we show that the automorphism groups of certain countable structures obtained using the Hrushovski amalgamation met\-hod are simple groups. This answers a question raised in \cite{MT} (Question (iii) of the Introduction there). The structures we consider are the `uncollapsed' structures of infinite Morley rank obtained by the ab initio construction in \cite{Hr} and the (unstable) $\aleph_0$-categorical pseudoplanes in \cite{Hrpp}. The simplicity of the automorphism groups of these follows from some quite general results which should be of wider interest and applicability. Although much of the intuition (and some of the motivation) behind these results is model-theoretic, the paper requires no particular knowledge of model theory.

\subsection{Background}

The methods we use have their origins in the paper \cite{L} of Lascar and it will be helpful to recall some of the results from there. Suppose $M$ is a countable saturated structure with a $\emptyset$-definable strongly minimal subset $D$ such that $M$ is in the algebraic closure of $D$. Denote the dimension function on $D$ coming from algebraic closure by $\dim$. Consider $G = \Aut(M/ \acl(\emptyset))$, the automorphisms of $M$ which fix every element (of $M^{eq}$) algebraic over $\emptyset$. Suppose $g \in G$ is \textit{unbounded} in the sense that for all $n \in \N$ there is a finite $X \subseteq D$ such that $\dim(gX/X) > n$. Then (\cite{L}, Th\'eor\`eme 2) the conjugacy class $g^G$ generates $G$. In particular, if all non-identity elements of $G$ are unbounded, then $G$ is a simple group. 

It is worth noting what Lascar's result says in the `classical' cases where $M = D$. If $M$ is a pure set, so $G$ is the full symmetric group  $\Sym(M)$, then $g \in G$ is bounded if and only if it is finitary. If $M$ is a countably  infinite dimensional vector space over a countable division ring $F$, then $G$ is the general linear group $\GL(\aleph_0, F)$ and $g \in G$ is bounded if and only if it has an eigenspace of finite codimension. So in these cases, Lascar's result implies the well known results, due to Schreier and Ulam \cite{SU} in the case of the symmetric group, and due to Rosenberg \cite{Ro} in the case of the general linear group, that $G$ modulo the bounded part is simple. If $M$ is an algebraically closed field of characteristic zero (and of countably infinite transcendence rank), then it can be shown that all non-identity automorphisms are unbounded, so in this case $G$ is simple (note that $\acl(\emptyset)$ is the algebraic closure of the prime field). 

Lascar's result is used directly in \cite{GT} to give examples of simple groups with a $BN$-pair which do not arise from algebraic groups. Ideas from Lascar's proof were used by Gardener \cite{G} to give an analogue of Rosenberg's result for classical groups of countably infinite dimension over finite fields.

%
%Topological methods are a key feature of Lascar's proof: the automorphism group $\Aut(M)$ is regarded as a topological group and arguments about Polish groups are used.

More recently, Lascar's ideas have been used in other contexts by Macpherson and Tent \cite{MT} and by Tent and Ziegler \cite{TZ}. A key feature in both of these papers is the use of a natural independence relation or notion of free amalgamation on $M$.  In \cite{MT}, $M$ is a homogeneous structure arising from a free amalgamation class of finite structures. Assuming $G = \Aut(M) \neq \Sym(M)$ is transitive on $M$, it is shown that $G$ is simple. The free amalgamation here can be viewed as giving a notion of independence on $M$, and \cite{TZ} formalizes this into the notion of a \textit{stationary independence relation} on $M$ (\cite{TZ}, Definition 2.1; cf. Definition \ref{sir} here). Generalising Lascar's notion of unboundedness, \cite{TZ} introduces the notion of $g \in \Aut(M)$ \textit{moving almost maximally} with respect to the independence relation (cf. Definition \ref{d26} here). It is shown (\cite{TZ}, Corollary 5.4) that in this case, every element of $G$ is a product of 16 conjugates of $g$.

\subsection{Main results} The paper contains two types of results. In Sections 2 and 3 we give general results along the lines of Lascar's result and the result of  Tent and Ziegler; in Sections 4 and 5 we apply these to the Hrushovski constructions. We first describe our generalisations of the results of \cite{L} and \cite{TZ}. As these require a number of technical definitions, we shall not state the results precisely in this introduction. 

In the results of \cite{MT} and \cite{TZ}, algebraic closure in $M$ is trivial. In Section 2 here we adapt the results of \cite{TZ} to remove this restriction. So $M$ will be a countable structure, $\cl$ an $\Aut(M)$-invariant closure operation on $M$ and we are interested in $G = \Aut(M/\cl(\emptyset))$. We define (Definition \ref{sir}) the notion of a stationary independence relation \textit{compatible with $\cl$} and observe (Theorem \ref{thmTZ}) that the above result of Tent and Ziegler also holds in this wider context. 

In Section 3, we assume that an integer-valued dimension function $d$ gives the closure $\cl^d$ and the independence notion $\ind^d$. This is the case in the Hrushovski construction which interests us, and of course is also the case in the almost strongly minimal situation of Lascar (where the closure is algebraic closure and dimension is given by Morley rank). We also assume a condition which we call \textit{monodimensionality} (Definition \ref{d35}) as a replacement for the assumption of almost strong minimality in Lascar's result. 
The main result here is Corollary \ref{cor39}: there is a natural notion of an automorpism being  $\cl^d$-bounded (Definition \ref{d312}); such automorphisms form a normal subgroup and if $g$ is not $\cl^d$-bounded, then every element of $G$ is the product of 96 conjugates of $g$ or its inverse. So this can be seen as a generalization of (\cite{L}, Th\'eor\`eme 2). A direct application of this result, together with work of Konnerth \cite{K}, shows that if $M$ is a countable, saturated differentially closed field of characteristic 0 and $F$ is the subfield of differentially algebraic elements of $M$, then $\Aut(M/F)$ is a simple group (see Example \ref{DCF}).

\medskip

The Hrushovski amalgamation constructions from \cite{Hr, Hrpp}  are of great importance in model theory and several related areas of mathematics. Here, we shall be concerned with the simplest forms of the construction where the basic ingredients are an integer-valued \textit{predimension} $\delta$ on a class $\C$ of finite structures for a fixed relational signature. For ease of notation, take natural numbers $r \geq 2$ and coprime $n,m \geq 1$ and work with a signature which has a single $r$-ary relation symbol $R$. Let $\C$ consist of finite structures in which $R$ is symmetric and only holds of distinct $r$-tuples. Thus we can consider the realisations of $R$ in a structure $A \in \C$ as a set $R[A]$ of $r$-subsets  of $A$. We let $\delta(A) = n\vert A \vert - m \vert R[A]\vert$. In order to state our results, we outline the relevant versions of the construction (more details can be found in Sections 4 and 5).

\medskip

\noindent (i) \textit{The uncollapsed case \cite{Hr}:\/} We let $\C_0 = \{ A \in \C : \delta(X) \geq 0 \mbox{ for all } X \subseteq A\}$. From this, we construct a countable structure $M_0$ with $\C_0$ as its collection of (isomorphism types of) finite substructures and which has an additional homogeneity property (see Theorem \ref{M0} for the precise definition). In general, this $M_0$ is $\omega$-stable of infinite Morley rank.

\medskip 

\noindent (ii) \textit{The $\omega$-categorical case \cite{Hrpp}:\/} In this version we have an increasing function $f : \R^{\geq 0} \to \R^{\geq 0}$ with $f(x) \to \infty$ as $x\to \infty$. We consider the class of structures $$\C_f = \{A \in \C_0 : \delta(X) \geq f(\vert X \vert)\,\, \forall X \subseteq A\}.$$ 
Under suitable assumptions on $f$ (see Assumptions \ref{ass}) one constructs a countable, $\omega$-categorical structure $M_f$ having $\C_f$ as its class of finite substructures and which has an additional homogeneity property (see Section 5.2). In general, $M_f$ will not be stable, though it can be supersimple (\cite{Hr97}).

\medskip

Note that we do not consider the original `collapsed' version of the construction from \cite{Hr} which produces structures of finite Morley rank,  as Lascar's result can be applied in this case. An analysis showing that there are no bounded automorphisms in some cases is carried out in \cite{ZGTh}.

\medskip

With the above notation, our main result for the uncollapsed case is Theorem \ref{413}:

\begin{theorem*} Suppose either that $r  = 2$ and $n > m$, or that $r \geq 3$ and $n \geq m$. Then $\Aut(M_0/\cl^d(\emptyset))$ is a simple group. In fact, if $g \in \Aut(M_0/\cl^d(\emptyset))$ is not the identity then every element of $\Aut(M_0/\cl^d(\emptyset))$ can be written as a product of 96 conjugates of $g^{\pm 1}$.
\end{theorem*}

\noindent(Here, $\cl^d(\emptyset) = \bigcup\{ X \subseteq_{fin} M_0 : \delta(X) = 0\}$.)

\medskip

For the $\omega$-categorical case, we have Theorem \ref{511}:

\begin{theorem*} Suppose Assumptions \ref{ass} hold. Suppose $M_f$ is monodimensional and $1 \neq g \in \Aut(M_f)$. Then every element of $\Aut(M_f)$ is a product of 192 conjugates of $g^{\pm 1}$. In particular, $\Aut(M_f)$ is a simple group.
\end{theorem*}

We believe that under the conditions of Assumption \ref{ass}, the structure $M_f$ should be monodimensional. However, we have only been able to verify this in the cases where $r\geq 3$ and $m=n=1$ (Example \ref{e511}), and, under some extra assumptions on $f$, where $r=2$, $n=2$ and $m=1$ (Example \ref{e512}).

\medskip

\subsection{Notation}  \label{not1} Throughout the paper, $M$ will denote a  first-order structure, which will usually be countable (though this will not be necessary for the purposes of some of the definitions). We will not distinguish notationally between a structure and its domain. We denote by $\Aut(M)$ the group of automorphisms of $M$ and if $X \subseteq M$, then $\Aut(M/X)$ is the subgroup consisting of automorphisms which fix every element of $X$. We also use an alternative notation for this: if $H \leq G$ is a group of permutations on $M$ and $X \subseteq M$, then  we let  $H_X= \{ h \in H : h(x) = x \mbox{ for all } x \in X\}$. If $a$ is a tuple of elements from $M$ then the \textit{$H$-orbit} of $a$ is $\{ h(a) : h \in H\}$. The $\Aut(M/X)$-orbit of $a$ is denoted by $\loc(a/X)$.

 If $A, B \subseteq M$ and $c = (c_1,\ldots, c_n)$ is a tuple in $M$, then we will often use notation such as $AB$ and $Ac$ in place of $A \cup B$ and $A\cup \{c_1,\ldots, c_n\}$. This notation will also be used in conjunction with a closure operation $\cl$ or dimension function $d$: so we might write $\cl(AB)$ or $\cl(A,B)$ instead of $\cl(A\cup B)$, and $d(c, A)$ or $d(cA)$ instead of $d(A \cup \{c_1,\ldots, c_n\})$. 
 
  We write $A \subseteq_{fin} B$ to indicate that $A$ is a finite subset of $B$. 
  
  If $g,h$ are elements of some group $G$, then $g^h$ denotes the conjugate $h^{-1}gh$ and $[g,h]$ is the commutator $g^{-1}h^{-1}gh = g^{-1}g^h$.

\medskip

\noindent\textit{Acknowledgements:\/} Several of the results given here appear in the PhD thesis of the second author \cite{ZGTh} with a slightly different presentation. Work on the paper was completed whilst the authors were participating in the trimester programme `Universality and Homogeneity' at the Hausdorff Institute for Mathematics, Bonn. The authors would like to thank the referee for a very thorough reading of an earlier version of this paper and for numerous helpful suggestions.

\section{Stationary independence relations}

In this section we use ideas from Lascar's paper \cite{L} to generalise some of the results from \cite{TZ}. Instead of giving complete proofs (which would involve reproducing large sections of \cite{TZ}), we will only sketch the modifications which are required to produce the generalisations. The treatment is mostly axiomatic: examples can be found in the applications later in the paper. 

\medskip
\begin{definition}\rm
Suppose $M$ is a  structure and $G = \Aut(M)$ is  its automorphism group. Let $\cl$ be a closure operation on $M$. We say that $\cl$ is \textit{invariant} if  for all $g \in G$ and $X \subseteq M$ we have $\cl(gX) = g(\cl(X))$. It is \textit{finitary} if   $\cl(X) = \bigcup \{\cl(Y) : Y \subseteq _{fin} X\}$ for all $X \subseteq M$.  We say that  $\cl$ \textit{subsumes definable closure} if whenever $X \subseteq_{fin} M$ and $a \in M$ is fixed by  $G_{\cl(X)}$, then $a \in \cl(X)$. 
\end{definition}

In the rest of this section, $M$ will be a countable structure and $\cl$ will be an invariant, finitary closure operation on $M$ which subsumes definable closure. 
We let  $\X = \{ \cl(X) : X \subseteq_{fin} M\}$ be the set of closures of finite subsets of $M$ and let $\F$ consist of all bijections $f: A \to B$ with $A, B \in \X$ which extend to automorphisms of $M$. 
We refer to the latter as \textit{partial automorphisms} of $M$. So of course, $\X$ is countable, but $\F$ need not be.

\medskip

Following Definition 2.1 of \cite{TZ}, we wish to define the notion of  an invariant stationary  independence relation  $\ind$ between elements of $\X$, or more generally between subsets of elements of $\X$, which is compatible with the closure operation $\cl$. More precisely we have the following modification of Definition 2.1 of \cite{TZ}.

\begin{definition}\label{sir} \rm Suppose $M$ is  a countable structure, $G = \Aut(M)$ and $\cl$ is an invariant, finitary closure operation on $M$ which subsumes definable closure. 
Let  $\X = \{ \cl(X) : X \subseteq_{fin} M\}$  and let $\F$ consist of all bijections $f: A \to B$ with $A, B \in \X$ which extend to elements of $G$.

We say that $\ind$ is a \textit{stationary independence relation compatible with $\cl$} if for $A, B, C, D \in \X$ and finite tuples $a, b$:

\begin{enumerate}

\item (Compatibility) We have $a \ind_b C \Leftrightarrow a \ind_{\cl(b)} C$ and    $$a\ind_B C \Leftrightarrow e\ind_B C \mbox{ for all } e \in \cl(a, B) \Leftrightarrow \cl(a, B) \ind_B C.$$

\item (Invariance) If $g \in G$ and $A\ind_B C$, then $gA\ind_{gB} gC.$

\item (Monotonicity) If $A \ind_{B} CD$, then $A\ind_B C$ and $A \ind_{BC} D$.

\item (Transitivity) If $A\ind_B C$ and $A \ind_{BC} D$, then $A\ind_B CD$.

\item (Symmetry) If $A\ind_B C$, then $C \ind_B A$.

\item (Existence) There is $g \in G_B$ with $gA \ind_B C$. 

\item (Stationarity)  Suppose $A_1, A_2, B, C \in \X$ with $B \subseteq A_i$ and $A_i \ind_B C$. Suppose $h : A_1 \to A_2$ is the identity on $B$ and $h \in \F$. Then there is some $k \in \F$ which contains $h \cup \id_{C}$ (where $\id_C$ denotes the identity map on $C$). 
\end{enumerate}
\end{definition} 

The prototypical example here is of course where $M$ is a (sufficiently homogeneous) stable structure with weak elimination of imaginaries, $\cl$ is algebraic closure and $\ind$ is non-forking independence. However, we will be interested in other examples, both where $M$ is unstable and where $\cl$ is larger than algebraic closure. We note the following.

\begin{lemma}  Suppose that the conditions of Definition \ref{sir} hold and $\ind$ is a stationary independence relation on $M$ compatible with $\cl$. Then: 
\begin{enumerate}
\item  for all $A \in \X$ and $X \subseteq_{fin} M$ we have $A \ind_X \cl(X)$; 
\item  if $A \in \X$ and  $b$ is a finite tuple in $M$ with $b\ind_A b$, then  $b \in A$.
\end{enumerate}
\end{lemma}

\begin{proof} (1) Let $B = \cl(X)$. By Existence, there is $g \in G_B$ with $gA \ind_B B$. By Invariance, it follows that $A \ind_B B$. By Compatibility, we then have $A \ind_X B$, as required.

(2) Suppose for a contradiction that $b \not\in A$. As $\cl$ subsumes definable closure, there is $g \in G_A$ with $b' = gb \neq b$. By Invariance, we therefore also have $b'\ind_A b'$. By Symmetry, (1) and Compatibility we have $b'\ind_{A, b'} b$. So by Transitivity (and Compatibility) we obtain $b'\ind_A b, b'$. Similarly $b\ind_A b,b'$. By Stationarity (and Compatibility), there is therefore $k \in G_{\cl(A, b, b')}$ with $kb = b'$. As $b \neq b'$, this is clearly impossible.
\end{proof}

\medskip

\begin{remark} \rm \label{sirvar} In the following, we will require a slightly more general version of Definition \ref{sir}. Suppose, as before, that $M$ is a countable structure and $\cl$ is an invariant, finitary closure operation on $M$ which subsumes definable closure. Let $G \leq \Aut(M)$ have the property that for all $A \in \X$, $G_A$ has the same orbits on finite tuples from $M$ as $\Aut(M/A)$ does. Then $\cl$ is $G$-invariant and `subsumes definable closure with respect to $G$' in the sense that if $A \in \X$ and $b \in M$ is fixed by $G_A$, then $b \in A$. We then say that $\ind$ is a \textit{stationary independence relation (with respect to $G$) compatible with $\cl$} if conditions (1)-(7) of Definition \ref{sir} hold, where $\F$ is the set of bijections $f : A\to B$ with $A, B \in \X$ which extend to elements of $G$.
\end{remark}

For the rest of this section, we shall assume that the conditions of Definition \ref{sir} hold and $\ind$ is a stationary independence relation on $M$ compatible with $\cl$. We use the notation from the definition. 

\medskip

As in Section 2 of Lascar's paper \cite{L}, we topologise $G = \Aut(M)$ by taking basic open sets of the form $O(f) = \{g \in G : g \supseteq f\}$, for $f \in \F$. It should be stressed that in general this is not the `usual' automorphism group topology (where pointwise stabilisers of finite sets form a base of  open neighbourhoods of the identity). It is complete metrizable, but not necessarily separable, so we cannot apply Polish group arguments directly to $G$.  However, as in \cite{L}, we will work in separable, closed subgroups to avoid this difficulty.

\medskip

Suppose $\S \subseteq \F$ and let 
\[ G(\S) = \{ g \in G : g \vert X \in \S \mbox{ for all } X \in \X\}.\]
Then $G(\S)$ is a closed subset of $G$ and if $\S$ is countable, $G(\S)$ is separable. Moreover, if $\S$ satisfies conditions (1-7) on page 241 of \cite{L}, then $G(\S)$ is a subgroup of $G$. Thus, if $\S$ is countable and satisfies these conditions then $G(\S)$ is a Polish subgroup of $G$. The conditions just say that $\S$: contains the identity maps; is closed under inverses, restrictions and compositions, and allows extension of domain (and codomain). It is clear that any countable $\S_0 \subseteq \F$ can be extended to a countable $\S$ satisfying these conditions. In particular, $G(\S)$ can be taken to include any desired countable subset of $G$.

\medskip

\medskip

\begin{lemma} \label{l25} Suppose $\S_0$ is a countable subset of $\F$. Then there is a countable $\S \subseteq \F$ with $\S_0 \subseteq \S$ such that $G(\S)$ is a group, for all $A \in \X$ we have that $G(\S)_A$ has the same orbits on finite tuples from $M$ as $G_A$,  and (in the terminology of Remark \ref{sirvar}) $\ind$ is a stationary independence relation (with respect to $G(\S)$) compatible with $\cl$.
\end{lemma}

\begin{proof} First, note that there is a countable $\S_1 \supseteq \S_0$ such that Lascar's conditions (1-7) hold and for all $B \in \X$, the group $G(\S_1)_B$  has the same orbits on finite tuples from $M$ as $G_B$. The latter will also be true if we enlarge $\S_1$ further and it implies that  the Existence condition in Definition \ref{sir} holds with respect to $G(\S_1)$. 

There is a countable $\S_2 \supseteq \S_1$ with the property that the Stationarity condition holds with respect to $G(\S_2)$. Alternating this with a step to ensure that Lascar's conditions (1-7) hold, we obtain, after  a countable number of steps, a countable set $\S\subseteq \F$ for which  (1-7) hold and the Stationarity condition holds with respect to $G(\S)$.
\end{proof}

The following definitions are adapted from \cite{TZ}. Only the first of these is needed to understand the statement of Theorem \ref{thmTZ} below and its subsequent applications; the other two definitions are used in its proof and are provided for the sake of completeness.

\begin{definition} \label{d26} \rm (1) (cf. Lemma 5.1 of \cite{TZ}) We say that $g \in G$ \textit{moves almost maximally} if for all $B \in \X$ and elements $a \in M$ there is $a'$ in the $G_B$-orbit of $a$ such that 
\[ a' \ind_B ga'.\]

(2) (cf. Definition 2.3 of \cite{TZ}) Suppose $x, y$ are finite tuples from $M$ (or are elements of $\X$) and $A, B \in \X$. We say that $x$ is \textit{independent from $y$ over $A; B$}, written $x \ind_{(A;B)} y$, if $x \ind_A By$ and $xA \ind_B y$. 

(3) (cf. Definition 2.5 of \cite{TZ}) 
 Suppose $g \in G$, $c$ is a finite tuple from $M$ and $B \in \X$. We say that $g$ \textit{moves $c$ maximally over $B$} if $c$ is independent from $gc$ over $B; gB$. We say that $g$ \textit{moves maximally} if for all $B \in \X$ and finite tuples $a$ there is $c$ in the $G_B$-orbit containing $a$ which is moved maximally over $B$ by $g$.
\end{definition}

Following the proof of Corollary 5.4 in \cite{TZ}, we then have: 

\begin{theorem} \label{thmTZ} Let $M$ be a countable structure and $\cl$ an invariant, finitary closure operation on $M$ which subsumes definable closure. Suppose that $\ind$ is a stationary independence relation on $M$ compatible with $\cl$ and that $G = \Aut(M)$ fixes every element of $\cl(\emptyset)$. If $g \in G$ moves almost maximally, then every element of $G$ is a product of 16 conjugates of $g$.
\end{theorem}

\begin{proof} Let $k \in G$ and let $\S_0 \subseteq \F$ be any countable set which contains  the restrictions of $k$ and $g$  to all elements of $\X$. Extend $\S_0$ to a countable set $\S$ as in Lemma \ref{l25}. So $g, k \in G(\S)$ and $G(\S)$ is a Polish group acting on $M$. Furthermore, $\ind$ is an invariant stationary independence relation \textit{with respect to this group}. 

For the rest of the proof only automorphisms in $G(\S)$ will be considered. The proof then just consists of checking that the argument in \cite{TZ} works. We make some remarks about various parts of this. 

\medskip

(1) We have the following  \textit{joint embedding property}. If $h_i : X_i \to Y_i$ are in $\S$ (for $i = 1,2$), then  there are $f \in G(\S)$ and $h \in \S$ with $f^{-1}h_1f , h_2 \subseteq h$. Indeed, by Existence we can assume (after applying a suitable $f\in G(\S)$) that $X_1, Y_1 \ind X_2, Y_2$. By Stationarity we can then extend $h_i$ to $g_i$ which is the identity on $X_j \cup Y_j$ (for $j \neq i$). Note that this uses the fact that $h_i$ fixes every element of $\cl(\emptyset)$. Then the product $g_1g_2$ extends $h_1$ and $h_2$, as required.

Once we have this, it follows that if $U, V$ are non-empty open subsets of $G(\S)$, then there is $f \in G(\S)$ such that $f(V) \cap U \neq \emptyset$. Thus Theorem 8.46 of \cite{K} applies as in the proof of Theorem 2.7 from Proposition 2.13 on p.294 of \cite{TZ} (this avoids the use of the hypothesis in 2.7 of \cite{TZ} that there is a dense conjugacy class in $G$).

\medskip

(2) The part of the proof in \cite{TZ} which requires the most adaptation is in the use of Lemma 3.6 in the proof of Proposition 3.4 there. So we give a reformulation of this lemma, and outline its proof. 

Suppose $g \in G$ moves maximally and let $X, Y \in \X$ with $gX = Y$. Suppose $X \subseteq W \in \X$ and $Y \subseteq Z \in \X$ are such that $W$ and $Z$ are independent over $X; Y$. Suppose $h : W \to Z$ is a partial automorphism (in $\S$) which extends $g\vert X$. Then there is $a \in  G_{\cl(XY)}$ such that $g^a(w) = h(w)$ for all $w \in W$ (where $g^a$ denotes the conjugate $a^{-1}ga$).

To see this, let $w$ be a finite tuple with $\cl(w) = W$ and let $w' \in \loc(w/X)$ be moved maximally over $X$ by $g$. So $w', gw'$ are independent over $X; Y$ and in particular $w'\ind_X Y$. Also $w\ind_X Y$, so by Stationarity there is $a_1 \in G_{\cl(XY)}$ with $a_1(w) = w'$. So $g^{a_1}$ moves $w$ maximally over $X$. Let $Z' = \cl(g^{a_1}(w))$. Thus $W \ind_{(X;Y)} Z'$. 

So $W, Y \ind_Y Z$ and $W,Y \ind_Y Z'$. We have partial automorphisms (in $\S$) $h: W \to Z$ and $h' : W \to Z'$ with $h'(w_1) = g^{a_1}(w_1)$ for $w_1 \in W$. Note that  $h(x) = h'(x)$ for $x \in X$. Let $k = h'h^{-1} : Z \to Z'$. Then $k(y) = y$ for all $y \in Y$. So by Stationarity, there is $a_2 \in G_{\cl(WY)}$ which extends $k$. It is then easy to check that $a = a_1a_2$ has the required properties.

\end{proof}

\section{Stationary independence relations with a dimension function}

Suppose $M$ is a countable structure and $G = \Aut(M)$. In this section we consider an independence relation arising from a dimension function on $M$.

\begin{definition} \label{dfdef} \rm We say that an integer-valued function $d$ defined on finite subsets (or tuples) from $M$ is a \textit{dimension function on $M$} if, for all $X, Y \subseteq_{fin} M$:
\begin{enumerate}
\item $d(gX) = d(X)$ for all $g \in G$;
\item $0 \leq d(X) \leq d(X\cup Y) \leq d(X) + d(Y) - d(X \cap Y)$.
\end{enumerate}

In this case, if  $X, Y \subseteq_{fin} M$, then  we define $d(X/Y) = d(XY) - d(Y)$ and for arbitrary $Z \subseteq M$ we let $d(X/Z) = \min(d(X/Y) : Y \subseteq_{fin} Z)$. We let $\cl^d(Z) = \{ a \in M : d(a/Z) = 0 \}$. \end{definition} 

We may assume in the above that $d(\emptyset) = 0$ (by replacing $d$ by the dimension function $d'(X) = d(X) - d(\emptyset)$) and we do this for the rest of the paper. 

If $d$ is an integer-valued dimension function on $M$ as above, then $\cl^d$ is an invariant, finitary closure operation on $M$. Let $\X = \{ \cl^d(X) : X \subseteq_{fin} M\}$ and for $A, B, C \in \X$ write $A\ind^d_B C \Leftrightarrow d(A/BC) = d(A/B)$ (where the dimension of an arbitrary set is the maximum of the dimensions of its finite subsets). We can also make the same definition if any of $A, B, C$ is a finite set (or tuple) from $M$. 

It is easy to check that $\ind^d$ satisfies (1-5) of Definition \ref{sir}.  In general, $\cl^d$ need not subsume definable closure, and $\ind^d$ need not satisfy condition (6) (Existence) of Definition \ref{sir}, so we shall assume these. We also wish to exclude the case where all finite sets have dimension zero (and once we do this, Existence will guarantee that there are finite sets of arbitrarily large dimension). For the rest of this section we make the following:

\begin{assumption}\label{dass} \rm Suppose $d$ is a non-zero, integer-valued dimension function on $M$ such that $\cl^d$ subsumes definable closure and $\ind^d$ satisfies Existence in Definition \ref{sir}.
\end{assumption}

If $\ind^d$ also satisfies (7) (Stationarity) in Definition \ref{sir}, we shall say that \textit{$\ind^d$ is stationary}.

We refer to $\cl^d$ and $\ind^d$ as \textit{$d$-closure} and \textit{$d$-independence}.

\begin{definition} \rm Suppose $b \in M$ and $A \in \X$. We say that $b$ is \textit{basic} over $A$ if $b \not\in A$ and whenever $A \subseteq C \in \X$ and $d(b/C) < d(b/A)$, then $b \in C$. 
\end{definition}

If $b$ is basic over $A \in \X$, then $b'$ is basic over $A$ for all $b' \in \loc(b/A)$ and we refer to $\orb(b/A)$ as a \textit{basic orbit} over $A$.

\begin{remarks}\rm As $d$ is integer-valued and non-negative, if $d(b/A) = 1$, then $b$ is basic over $A$. It is clear that if $b \not\in A$ there is some $A \subseteq C \in \X$ such that $b$ is basic over $C$. In general,  there is no reason why there should be such a $C$ with $d(b/C) = 1$, which is why we are working with this notion.
\end{remarks}

Suppose $A \in \X$ and $D \subseteq M$ is such that the elements of $D\setminus A$ are basic over $A$. We claim that  $d$-closure over $A$ on $D$ gives a pregeometry on $D$. So we need to verify the exchange condition: if $c_1, c_2 \in D$ and $c_1 \in \cl^d(A, c_2) \setminus A$, then $c_2 \in \cl^d(c_1, A)$. By assumption, $d(c_1, c_2/A) = d(c_2/A)$. So $d(c_2/A c_1) = d(c_1, c_2 /A) - d(c_1/A) < d(c_2/A)$, whence $d(c_2/A c_1) = 0$ (as $c_2$ is basic over $A$), as required.

If $X \subseteq D$ is finite, we write $\dim_A(X)$ for the dimension of $X$ with respect to this pregeometry. It is easy to show that if $c_1,\ldots, c_r \in D$ then $\dim_A(c_1,\ldots, c_r) = r$ if and only if  $c_1,\ldots, c_r$ are $d$-independent over $A$ (meaning that $d(c_1,\ldots, c_r/A) = \sum_i d(c_i/A)$).

Note that if $B \in \X$ contains $A$ then all elements of $D\setminus B$ are basic over $B$, so we can also consider $\dim_B$ on $D$.

\begin{definition} \label{d35} \rm We say that $M$ (with dimension function $d$) is \textit{monodimensional} if for every $A \in \X$ and basic $G_A$ -orbit $D$ there is $A \subseteq B \in \X$ with  $M = \cl^d(B,D)$.
\end{definition}

\textit{Remark:\/} The terminology is chosen  by association with the model-theoretic notion of unidimensionality. However, the two notions do not coincide as the structures we consider in the next section are not unidimensional, which is why we feel obliged to invent a different terminology.

\medskip

If $\ind^d$ is stationary, we can check monodimensionality on a single basic orbit.

\begin{lemma}\label{monod} Suppose $\ind^d$ is stationary, $A \in \X$ and $D$ is a basic $G_A$-orbit.
\begin{enumerate}
\item If $A \subseteq B \in \X$, then $D\setminus B$ is a basic $G_B$-orbit.
\item If $\cl^d(A, D) = M$, then $M$ is monodimensional.
\item Suppose that for every $c \in M\setminus A$ there is a finite tuple $b$ of elements of $D$ such that $c \nind^d_A b$. Then $M$ is monodimensional.
\end{enumerate}
\end{lemma}

\begin{proof} (1) If $b_1, b_2 \in D \setminus B$ then $b_i \ind^d_A B$. So by Stationarity, $b_1, b_2$ are in the same $G_B$-orbit.

(2) Suppose $C \in \X$ and $E$ is a basic $G_C$-orbit. Let $B' = \cl^d(A, C)$. By (1), $D\setminus B'$ and $E \setminus B'$ are basic $G_{B'}$-orbits. We have $M = \cl^d(B', D\setminus B')$   and want to show that  $\cl^d(B,E) = M$ for some $B' \subseteq B \in \X$. In other words, we may assume from now on that $C  = B' = A$. 

Let $e \in E$ and choose $c_1,\ldots, c_r \in D$ independent over $A$ with $e \in \cl^d(c_1,\ldots, c_r, A)$ and $r$ as small as possible. As $\cl^d$ over $A$ gives a pregeometry on $D \cup E$, we may assume (by the exchange condition) that $c_1 \in \cl^d(e, c_2, \ldots,  c_r, A)$. Let $B= \cl^d(c_2,\ldots, c_r, A)$. So $c_1 \in \cl^d(B, e)\setminus B$. Thus, $\cl^d(B, E)$ contains a point $c_1$ of the basic $G_B$-orbit $D\setminus B$. It is clearly $G_B$-invariant, and therefore contains the whole of $D\setminus B$. It follows that $\cl^d(B, E) \supseteq \cl^d(B, D\setminus B) = M$, as required.

(3) We show by induction on $r = d(c/A)$ that $c \in \cl^d(A, D)$. The induction is over all $A, D$. If $r = 0$, there is no problem. Otherwise we can find a finite tuple $e$ in $D$ with $c\nind_A^d e$. So $d(c/A, e) < d(c/A)$. Let $B = \cl^d(A, e)$. By induction and (1) there is a finite tuple $e'$ in $D\setminus B$ such that $c \in \cl^{d}(B, e')$, as required.
\end{proof}

The following notion of boundedness is less natural than Lascar's. We shall connect it with a more natural notion later in this section.

\begin{definition} \rm Suppose $A \in \X$. We say that $h \in G$ is  \textit{unbounded} \textit{over} $A$ if  for all  $A \subseteq C \in \X$ and $b \in M$ which is basic over $C$, there is $b'\in \loc(b/C)$ with $hb' \ind^d_C b'$ (or equivalently, $b' \not\in \cl^d(C, hb')$). We say that $h$ is \textit{unbounded} if it is unbounded over some $A \in \X$, otherwise, it is \textit{bounded}.\end{definition}

Note that if $h$ is unbounded over $A$ and $A \subseteq B \in \X$, then $h$ is unbounded over $B$.

\begin{proposition}\label{thm1} Suppose $A \in \X$ is such that there is a $G_A$-invariant set $D$  where the elements of $D\setminus A$ are basic over $A$ and $\cl^d(D, A) = M$. Let $h \in G$ be unbounded over $A$. 
\begin{enumerate}
\item If $A \subseteq B \in \X$ and $c$ is a finite tuple in $M$, then there is $c' \in \loc(c/B)$ with $hc' \ind^d_B c'$.
\item If $\ind^d$ is stationary, and $h \in G_A$, then every element of $\Aut(M/A)$ is a product of 16 conjugates of $h$.
\end{enumerate}
\end{proposition}

\begin{proof} (1) First, we show that this holds for $c$ an $n$-tuple of elements of $D$ with $\dim_B(c) = n$. If $n=1$, this is just the definition of unboundedness of $h$. If $n > 1$ and $c = (c_1,\ldots, c_n)$ then write $e = (c_1,\ldots, c_{n-1})$. Inductively, there is $e' \in \loc(e/B)$ with $he'\ind^d_B e'$. Let $f'$ be such that $c' = (e',f') \in \loc(c/B)$, $f' \not\in \cl^d(h^{-1}e', h^{-1}B, e')$ and (using the unboundedness) $f' \not\in \cl^{d}(e',B, he', hf')$. From the second of these, $hf' \not\in \cl^d(e',B, he')$ and so, from the third, $\dim_B(f', hf', he', e') = 2 + \dim_B(he', e') = 2+2(n-1) = 2n$. Thus $\dim_B(c', hc') = 2n$ and therefore $hc'\ind_B^d c'$, as required.

Now suppose $b \in M$. By assumption on $D$, there is a tuple $c \in D^n$ such that $b \in \cl^d(c, B)$. Clearly we can take $c$ to be $d$-independent over $B$. Let $B_1 = \cl^d(B, hB)$. By Extension, there is $b_1c_1 \in \loc(bc/B)$ with $c_1\ind^d_B B_1$. 

By the above, we can find  $b_2c_2 \in \loc(b_1c_1/B_1)$ with $c_2\ind^d_{B_1} hc_2$. Then $b_2 \ind^d_{B_1} hc_2$. Moreover, as $b_2 \in \cl^d(c_2, B)$ we have $hb_2 \in \cl^d(hc_2, hB) \subseteq \cl^d(hc_2, B_1)$. Thus $b_2\ind^d_{B_1} hb_2$. 

We also have $c_2 \ind^d_B B_1$, so $b_2 \ind^d_B B_1$, therefore $b_2 \ind^d_B hb_2$. As $b_2 \in \loc(b/B)$, this completes the proof of (1).

(2) This follows from (1) and Theorem \ref{thmTZ}.
\end{proof}

\begin{remark}\rm  Suppose $c\in M$ and $B \subseteq M$.  If $h$ is any automorphism of $M$, then $h(\loc(c/B))$ is the translate of this $G_B$-orbit by $h$. It is a $G_{hB}$-orbit, and depends only on the restriction of $h$ to $B$ (for the latter point, note that if $h' \in G$ has the same restriction to $B$ as $h$, then $hc, h'c$ lie in the same $G_{hB}$-orbit, because $h'h^{-1} \in G_{hB}$). So the notation $h(\loc(c/B))$  also makes sense if $h$ is a partial automorphism with $B$ in its domain.
\end{remark}

\begin{theorem}\label{thm2} Suppose $\ind^d$ is stationary and  $A \in \X$ is such that there is a $G_A$-invariant set $D$  where the elements of $D\setminus A$ are basic over $A$ and $\cl^d(D, A) = M$.   Suppose $g \in \Aut(M/\cl^d(\emptyset))$ is an unbounded automorphism of $M$. Then every element of $\Aut(M/\cl^d(\emptyset))$ is a product of 96 conjugates of $g^{\pm 1}$. 
\end{theorem}

\begin{proof} By enlarging $A$ if necessary, we can assume that $g$ is unbounded over a subset of $A$. We first show that there is $\h \in \Aut(M/\cl^d(\emptyset))$ such that the commutator $g_1 = [g,\h] = g^{-1}\h^{-1}g\h$ is in $G_A$ and is unbounded (over $A$). We build $\h$ by back-and-forth as the union of a chain of partial automorphism (with domains and images in $\X$).

Note that if $h$ is a partial automorphism which fixes all points of $A \cup gA$, then $g^{-1}h^{-1}gh(a) = a$ for all $a \in A$. So we start the construction of $\h$ with such  a partial automorphism. There is no problem extending this to an automorphism, the issue is  to ensure the unboundedness of $g_1$. We enforce this in the `forth' step in the construction. 

\medskip

Suppose that the partial automorphism $h$ has been defined and $B = \dom(h)$. Suppose $C \subseteq B$, $C \in \X$ and $a$ is basic over $C$. We want to find $a' \in \loc(a/C)$ so that (once $\h$ is defined) $g_1a'\ind_C a'$, that is, $a' \not\in \cl^d(g_1a', C)$. It will suffice to do this with $C = B$.

So suppose that $a$ is basic over $B$. We may assume (by Existence) that $a \not\in \cl^d(B, gB)$. By unboundedness of $g$ there is $b \in h(\loc(a/B))$ such that $gb \ind^d_{hB} b$. Extend $h$ to $h'$ with $h'a = b$. 

By Existence, there is $c \in h'^{-1}(\loc(gb/hB, b))$ with $c \ind^d_{B, a} gB, ga$. Extend $h'$ to $h''$ with $h''(c) = gb$. As $gb \ind^d_{hB} b$ we have (applying $h''^{-1}$) that $c \ind^d_B a$. Thus, by Transitivity, $c \ind^d_B gB, ga$, so $c \ind^d_{B, gB} ga$.  Then $g^{-1}c \ind^d_{g^{-1}B, B} a$. As $a$ is basic over $B$ and $a \not\in \cl^d(B, g^{-1}B)$, we have $g^{-1}B \ind_B^d a$.  It follows that $g^{-1}c \ind^d_B a$, that is, 
\[ g^{-1}h''^{-1}gh''a \ind^d_B a\]
as required.

\medskip

It now follows from Proposition \ref{thm1} that every element of $G_A$ is a product of 32 conjugates of $g^{\pm 1}$. Thus, to prove the theorem, it will suffice to show that $\Aut(M/\cl^{d}(\emptyset))$ is a product of 3 conjugates of $H_1 = G_A$.

\medskip

By Existence, there is $A' \in \loc(A/\cl^d(\emptyset))$ with $A' \ind^d A$. So $H_2 = G_{A'}$ is a conjugate of $H_1$. Let $k \in \Aut(M/\cl^d(\emptyset))$. By Existence again, there is $f_1 \in H_1$ with $f_1A' \ind^d A, kA$. By Stationarity, there is $f_2 \in \Aut(M/f_1 A')$ with $f_2 \vert A = k \vert A$. Thus $f_2^{-1}k \in H_1$ and so $k \in f_2H_1$. But $f_2 \in f_1H_2f_1^{-1}$, so $k \in H_1H_2H_1$, as required.
\end{proof}

We now give a more natural interpretation of boundedness when $M$ is monodimensional. Note that the following does not require stationarity of $\ind^d$. 

\begin{proposition}\label{unbdd} Suppose $M$ is monodimensional and suppose $g \in G$ is bounded. Then there is $E \in \X$ such that $g(B) = B$ for all $B \in \X$ which contain $E$.
\end{proposition}

\begin{proof} 
There is $C \in \X$ and a basic $b$ over $C$ such that for all $b' \in \loc(b/C)$ we have $b' \in \cl^d(C, gb')$, so $g^{-1}b' \in \cl^d(g^{-1}C, b')$. By extending $C$ if necessary, we can assume by monodimensionality that $\cl^d(C, \loc(b/C)) = M$.  There are $b_1,\ldots, b_k \in \loc(b/C)$ with $g^{-1}C \subseteq \cl^d(C, b_1,\ldots, b_k) = E$. So 
$$g^{-1}E = \cl^d(g^{-1}C, g^{-1}b_1,\ldots, g^{-1}b_k) \subseteq \cl^d(g^{-1}C, b_1,\ldots, b_k) \subseteq E.$$
As $d(E) = d(g^{-1}E)$ we obtain $g^{-1}E = E$. Let $b_1 \in \loc(b/C)$ be such that $b_1\ind^d_C E$. Then $b_1$ is basic over $E$ and for all $b' \in \loc(b_1/E)$ we have that $g^{-1}$ stabilizes $\cl^d(E, b')$ (setwise) and therefore $g$ stabilizes it also. Note that this implies that if $\b$ is a tuple of elements from $\loc(b_1/E)$, then $g$ stabilizes $\cl^d(E, \b)$ setwise.

Now, given any $B \supseteq E$ in $\X$ we can find a tuple $\b$ of elements of $\loc(b_1/E)$ such that $B_1 = \cl^d(E, \b) \supseteq B$. Then (by Extension) we can find $B_2 \in \loc(B_1/B)$ with $B_2 \ind^d_B B_1$: in particular $B_1\cap B_2 = B$. By the previous paragraph, $g$ stabilizes both $B_1$ and $B_2$, so $gB = B$.
\end{proof}

\begin{definition}\label{d312} \rm We say that $g \in Aut(M)$ is \textit{$\cl^d$-bounded} if there is some $E \in \X$ such that $g$ stabilizes setwise all $B \in \X$ which contain $E$.
\end{definition}

It is easy to see that the $\cl^d$-bounded automorphisms form a normal subgroup of $\Aut(M)$. The following follows from the above two results and can be seen as a generalisation of Theorem 2 of \cite{L} (the almost strongly minimal case, where there is a strongly minimal set definable over the empty set). 

\begin{corollary} \label{cor39} Suppose $\ind^d$ is stationary and $A \in \X$ is such that there is a basic $\Aut(M/A)$-orbit $D$ with $\cl^d(A, D) = M$. Suppose $g \in \Aut(M/\cl^d(\emptyset))$ is not $\cl^d$-bounded. Then every element of $\Aut(M/\cl^d(\emptyset))$ is a product of 96 conjugates of $g^{\pm 1}$.\hfill $\Box$
\end{corollary}

\begin{example} \label{DCF} \rm Suppose $M$ is a countable, saturated differentially closed field of characteristic 0. If $a$ is a tuple of elements of $M$, let $d(a)$ denote the differential transcendence degree of $a$ over $\emptyset$. This is a dimension function on $M$ which satisfies the conditions of Assumption \ref{dass}. It follows from (\cite{konn}, Corollary 2.6) that $\ind^d$ is  stationary.  The elements of differential transcendence degree 1 form a single orbit $D$ under $G = \Aut(M/\cl^d(\emptyset))$ and clearly $\cl^d(D) = M$, so Corollary \ref{cor39} applies. By (\cite{konn}, Proposition 2.9), the only $\cl^d$-bounded automorphism of $M$ is the identity, so $\Aut(M/\cl^d(\emptyset))$ is a simple group. In fact, because we can use Proposition  \ref{thm1} with $A = \cl^d(\emptyset)$, if $1 \neq g \in G$, then every element of $G$ is a product of 16 conjugates of $g$. 

\end{example}

\section{The ab initio Hrushovski constructions}

\subsection{The structures} \label{Sec41}

The Hrushovski construction which originated in \cite{Hr} admits many extensions and variations and can be presented at various levels of generality. But to fix notation, we consider the following basic case and comment on generalizations later. The article \cite{W0} is a convenient general reference for these constructions.

Suppose $r \geq 2$ and  $m,n \geq 1$ are fixed coprime integers. We work with the class $\C$ of finite \textit{$r$-uniform hypergraphs}, that is, structures in a language with a single $r$-ary relation symbol $R(x_1,\ldots,x_r)$ whose interpretation is invariant under permutation of coordinates and satisfies $R(x_1,\ldots,x_r) \to \bigwedge_{i<j} (x_i \neq x_j)$.   If $ B \in \C$ consider the \textit{predimension} 
$$ \delta(B) = n\vert B\vert - m\vert R[B] \vert$$
where $R[B]$ denotes the set of hyperedges on $B$ (i.e $ \{ \{b_1,\ldots, b_r\} : B \models R(b_1,\ldots,b_r)\}$). For $A\subseteq B$, we write $A \leq B$ iff for all $A \subseteq B' \subseteq B$ we have $\delta(A) \leq \delta(B')$, and let $\C_0 = \{ B \in \C: \emptyset \leq B\}$. The following is standard (cf. (\cite{Hr}, Lemma 1), for example).

\begin{lemma}\label{411} Suppose $A, B \subseteq C \in \C$. Then: 
\begin{enumerate}
\item $\delta(A\cup B) \leq \delta(A)+\delta(B)-\delta(A\cap B)$;
\item if $A \leq B$ and $X \subseteq B$, then $A\cap X \leq X$;
\item if $A \leq B \leq C$, then $A \leq C$.
\end{enumerate}
\end{lemma}

We let $\bar{\C}_0$ be the set of structures all of whose finite substructure are in $\C_0$. If $C \subseteq B \in \bar{\C}_0$ we write $C \leq B$ iff $X\cap C \leq X$ for all finite $X \subseteq B$. (This agrees with what was previously defined, by the above lemma). If $A, B \subseteq_{fin} C \in \C_0$ then we define $\delta(A/B) = \delta(A\cup B)- \delta(B)$.  Note that this is equal to $\vert A\setminus B\vert - \vert R[A\cup B]\setminus R[B]\vert$ and this makes sense for arbitrary $B$ (allowing the value  $-\infty$, if necessary). Then $B \leq A\cup B$ iff $\delta(A' /B) \geq 0$ for all $A' \subseteq A$.

 The class $\bar{\C}_0$ has the following  amalgamation property: suppose $B, C \in \bar{\C}_0$ have a common substructure $A$ and $A \leq B$. Then the \textit{free amalgam} $F = B \coprod_A C$ of $B$ and $C$ over $A$,  consisting of the disjoint union of $B$ and $C$ over $A$ with only the relations on $B$ and on $C$, is in $\bar{\C}_0$ and $C \leq F$. Using this and a standard Fra\"{\i}ss\'e-style construction, we obtain the following well-known result, which is sometimes referred to as the \textit{ab initio} case of the Hrushovski construction:

\begin{theorem} \label{M0} There is a countable structure $M_0 \in \bar{\C}_0$, unique up to isomorphism,  having the properties: $M_0$ is a union of a chain of finite $\leq$-substructures;  if $X \leq M_0$ is finite and $X \leq A \in \C_0$, then there is an embedding $\alpha: A \to M_0$ which is the identity on $X$ and $\alpha(A) \leq M_0$. Moreover, if $A_1, A_2 \leq M_0$ are finite and $h : A_1 \to A_2$ is an isomorphism, then $h$ extends to an automorphism of $M_0$.\hfill$\Box$
\end{theorem}

\noindent The structure $M_0$ is called the \textit{generic structure} for the class $(\C_0, \leq)$. We refer to the property in the `Moreover' statement  as \textit{$\leq$-homogeneity} of $M_0$. It is easy to see that every countable structure in $\bar{\C}_0$ can be embedded as a $\leq$-substructure of $M_0$. For the rest of this section, $M_0$ will denote the generic structure in Theorem \ref{M0}. 

As usual, we have two closure operations and a dimension function on $M_0$ (indeed, on any structure in $\bar{\C}_0$). If $X$ is a finite subset of $M_0$, there is a smallest  subset $Y$ with $X \subseteq Y \leq M_0$. This $Y$ is finite and we denote it by $\cl_0(X)$. The dimension $d(X)$ of $X$ (in $M_0$)  is defined to be $\delta(\cl_0(X))$. This is a dimension function on $M$ in the sense of Definition \ref{dfdef}. As in the previous section, the $d$-closure of $X$ is $\cl^d(X) = \{ a \in M_0 : d(X \cup \{a\}) = d(X)\}$ and in general, this will not be finite. We shall eventually show (Corollary \ref{39}) that the assumptions \ref{dass} hold for $(M_0; d)$ and $\ind^d$ is stationary.  Let $\X = \{\cl^d(X) : X \subseteq_{fin} M_0\}$.

As in the previous section, for tuples $a, b, c$ in $M_0$ we define $a\ind_b^d c$ to mean $d(a/b) = d(a/bc)$; similarly for sets in $\X$. This is not the same as non-forking independence. The following is well-known.

\begin{lemma} \label{421} If $A, B, C \in \X$ then $A \ind^d_B C$ if and only if the following three conditions hold: $\cl^d(AB) \cap \cl^d(BC) = B$; $\cl^d(AB)$, $\cl^d(BC)$ are freely amalgamated over $B$; and $\cl^d(AB) \cup \cl^d(BC) \leq M_0$.\hfill $\Box$
\end{lemma}

\subsection{Extending the homogeneity} 

 We will show that if  $A_1, A_2 \in \X$ and $h : A_1 \to A_2$ is an isomorphism, then $h$ extends to an automorphism 
of $M_0$. 

We need the following notion from \cite{Hr}. Suppose $Z \subset Y \in \bar{\C}_0$ and $Y \setminus Z$ is finite. We say that the extension $Z\subset Y$ is \textit{simply algebraic} if $\delta(Y/Z) = 0$ and whenever $Z \subset Z_1 \subset Y$, then $\delta(Y/Z_1) < 0$.  So  $Z \leq Y$, but $Z_1 \not\leq Y$ for all $Z \subset Z_1 \subset Y$. We write \textit{sa} for simply algebraic. The extension is \textit{minimally} simply algebraic (\textit{msa}) if the extension $Z_0 \subset Z_0\cup(Y\setminus Z)$ is not simply algebraic for all proper subsets $Z_0$ of $Z$. In this case $Z$ is finite and more generally, if $Z \subset Y$ is simply algebraic, there is finite subset $Y_1$ of $Y$ which contains $Y \setminus Z$ and is such that $Y_1\cap Z \subset Y_1$ is msa. Moreover, $Y$ is the free amalgam of $Z$ and $Y_1$ over $Z_1 = Y_1 \cap Z$. (In fact, $Z_1$ consists of the points in $Z$ which are in some $R$-relation containing a point of $Y\setminus Z$.)  In this case, we say that $Y$ has \textit{base} $Z_1$ and \textit{type} $(Z_1, Y_1)$ over $Z$. 

If $A \leq M_0$ and $B \subseteq M_0$ is an sa extension of $A$, then $B \leq M_0$. Moreover, any collection $\{ B_i : i \in I\}$ of (distinct) sa extensions of $A$ in $M_0$ is in free amalgamation over $A$ and $\bigcup_{i \in I} B_i \leq M_0$ (Lemma 2 of \cite{Hr}). If $Z_1 \subseteq A$ and $Z_1 \subset Y_1$ is msa, then the \textit{multiplicity} $\m(Z_1,Y_1/A)$ is the number of distinct minimal extensions of $A$ of type $(Z_1,Y_1)$ in $M_0$. So this is the maximum cardinality of $\{B_i : i \in I\}$ where each $B_i$ is a sa extension of $A$ of type $(Z_1, Y_1)$. Note that $\cl^d(A) = A$ iff each such multiplicity is zero. Indeed, $\cl^d(A)$ is the union of all subsets of $M_0$ which can be obtained from $A$ by a finite chain of successive sa  extensions. The free amalgamation property for $\C_0$ shows that if $A$ is finite, then all multiplicities over $A$ are infinite.

\begin{definition} \rm Suppose $A_1, A_2 \leq M_0$ and $k : A_1 \to A_2$ is an isomorphism. We say that $k$ is \textit{potentially extendable} if for every $Z_1 \subseteq A_1$ and msa $Z_1 \subset Y_1$ we have $\m(Z_1,Y_1/A_1) = \m(Z_2, Y_2/A_2)$, where $Z_2 = k(Z_1)$, and $k \vert Z_1$ extends to an isomorphism between $Y_1$ and $Y_2$. 
\end{definition}

Evidently, if $k$ as above extends to an automorphism of $M_0$, then $k$ is potentially extendable. Moreover, there are isomorphisms $k : A_1 \to A_2$ with $A_i \leq M_0$ which are not potentially extendable. 

\begin{lemma}\label{pe} If $A_1, A_2 \leq M_0$ are such that $d(A_i)$ is finite and $k : A_1 \to A_2$ is potentially extendable, then $k$ can be extended to an automorphism of $M_0$.
\end{lemma}

\begin{proof} For $i = 1,2$, let $A_i'$ be the union of all sa extensions of $A_i$ in $M_0$. By the above, $A_i' \leq M_0$ and $A_i'$ is the free amalgam over $A$ of the various sa extensions. So by the condition on the multiplicities, $k$ extends to an isomorphism $k' : A_1' \to A_2'$. 

We claim that $k'$ is potentially extendable. Indeed, suppose $Z_1 \subseteq A_1'$ is finite and $Z_1 \subset Y_1$ is msa. If $Z_1 \subseteq A_1$ then by construction of $A_1'$ we have $\m(Z_1,Y_1/A_1') = 0$. So it will suffice to show that if $Z_1 \not\subseteq A_1$ then there are only finitely many copies of $Y_1$ over $Z_1$ in $A_1'$ (because it then follows that $\m(Z_1,Y_1/A_1')$ is infinite, and the same will be true for the corresponding msa extension of $k'(Z_1)$ over $A_2'$). 

To see this, note that as $A_1'$ is a free amalgam over $A_1$, any point in $A_1'\setminus A_1$ is contained in only finitely many instances of the relation $R$. But, in any msa extension, every point in the base is in some instance of the relation $R$ which also contains a non-base point. As any two msa extensions with the same base are disjoint over the base, it follows that $Z_1$ is the base of only finitely many msa extensions contained in $A_1'$.

This shows that $k'$ is potentially extendable, so we can repeat the argument and adjoin to  $A_1'$ all sa extensions of $A_1'$ and extend $k'$. Continuing in this way, we see that we can extend $k$ to $h : B_1 \to B_2$, where $B_i = \cl^d(A_i)$.  Evidently $h$ is potentially extendable (as all multiplicities over its domain and image are zero).

Now, suppose we have $c \in M_0$. It will be enough to show how to extend $h$ to a potentially extendable map which has $c$ in its domain (for then we can proceed by a back-and-forth argument to build up an automorphism extending the original $k$). We may assume $c \not\in B_1$. Let $S_0 \subseteq B$ be finite and such that $\cl^d(S_0) = B_1$ and let $S = \cl_0(c, S_0)\cap B_1$. Then $S \leq M_0$ is finite and $\cl_0(c, S) \cap \cl^d(S) = S$. Furthermore, $C = \cl_0(c, S)$ and $B_1$ are freely amalgamated over $S$, and $C \cup B_1 \leq M_0$. 

Let $T = h(S)$ and $T \leq D \in \C_0$ be such that $h \vert S$ extends to an isomorphism $C \to D$. We claim that we can find a copy $D_1$ of $D$ over $T$ such that $D_1, B_2$ are freely amalgamated over $T$ and $D_1 \cup B_2 \leq M_0$. In fact, take any copy $D_1 \leq M_0$ of $D$ over $T$ in $M_0$: this exists, by the characteristic property in Theorem \ref{M0}. We have $cl^d(T) \cap D_1 = T$ (because the same is true of $S \leq C$), so $D_1 \cap B_2 = T$. The other properties follow as $d(D_1/T) = d(D_1/B_2)$. 

So now we can extend $h$ to $h' : B_1 \cup C \to B_2 \cup D$ and to finish, we need to show that $h'$ is potentially extendable. But this is a similar argument to what was done previously. If $Z_1 \subset B_1\cup C$ and $Z_1 \subset Y_1$ is msa, then either $Z_1 \subseteq B_1$, in which case $\m(Z_1,Y_1/B_1) = 0$, or $Z_1 \cap (C\setminus B_1) \neq \emptyset$. But points in $C\setminus B_1$ are in only finitely many relations within $B_1 \cup C$, so in this latter case $B_1 \cup C$ contains only finitely many copies of $Y_1$ over $Z_1$. Thus $\m(Z_1,Y_1/B_1)$ is infinite. The same argument also holds with $B_2$ and $D_1$, so we are finished.
\end{proof}

\begin{lemma}\label{fe} Suppose $A, C \in \X$ are freely amalgamated over $B = A \cap C$ and $A \cup C \leq M_0$. Then for every msa $Z \subset Y$ with $Z \subseteq A \cup C$ and $Z \not\subseteq A$ and $Z\not\subseteq C$, there are only finitely many copies of $Y$ over $Z$ in $A\cup C$. In particular, $\m(Z,Y/A\cup C)$ is infinite.
\end{lemma}

\begin{proof}  The proof of Hrushovski's algebraic amalgamation lemma (Lem\-ma 3 of \cite{Hr}) shows that there are at most $\delta(Z)$ copies of $Y$ over $Z$ which are contained in $A \cup C$. 
\end{proof}

\begin{corollary} \label{cor1} We have the following additional homogeneity properties of $M_0$.
\begin{enumerate}
\item ($d$-homogeneity:\/) Suppose $A_1, A_2 \in \X$ and $h : A_1 \to A_2$ is an isomorphism. Then $h$ extends to an automorphism of $M_0$.
\item ($d$-stationarity:\/) Suppose $A_1, A_2, C \in \X$. Suppose that for each $i$ we have that $A_i \cup C \leq M_0$ and $A_i$, $C$ are freely amalgamated over $B = A_i \cap C$. If $h : A_1 \to A_2$ is an isomorphism which is the identity on $B$, then $h$ extends to an automorphism of $M_0$ which fixes every element of $C$ pointwise.
\end{enumerate}
\end{corollary}

\begin{proof} (1) As the $A_i$ are $d$-closed, $h$ is potentially extendable. So by Lemma \ref{pe}, it extends to an automorphism of $M_0$.

(2) Let $k : A_1 \cup C \to A_2 \cup C$ be the union of $h$ with the identity map on $C$. As $A_i, C$ are freely amalgamated over their intersection $B$ (for $i = 1,2$), this is an isomorphism. By Lemma \ref{fe}, it is potentially extendable. So by Lemma \ref{pe}, it extends to an automorphism of $M_0$.
\end{proof}

\begin{corollary} \label{39} The dimension function $d$ on $M_0$ satisfies Assumption \ref{dass} and the relation $\ind^d$ is  stationary.
\end{corollary}

\begin{proof} First, we verify the Existence property. Given $A, B, C \in \X$ we need to show that there is $g \in G_B$ with $gA \ind^d_B C$. By taking $d$-closures over $B$, we may assume that $B \subseteq A, C$. Let $F$ be the free amalgam of $A, C$ over $B$ and let $A'$ denote the copy of $A$ inside $F$. So there is an isomorphism  $h : A \to A'$  which is the identity on $B$. By the construction of $M_0$ we can assume that $F \leq M_0$. Then $A' \ind^d_B C$ and $h$ extends to an automorphism $g$ of $M_0$ by $d$-homogeneity (Corollary \ref{cor1}). Note that in this, we can take $C = A$ and so if $a \in A \setminus B$, then $a$ is not fixed by $\Aut(M_0/B)$. Thus, $\cl^d$ subsumes definable closure. 

Finally, we note that Corollary \ref{cor1} and Lemma \ref{421} give the stationarity of $\ind^d$.
\end{proof}

\subsection{Bounded automorphisms} We shall show that, under a mild restriction on the parameters $n, m, r$, the structure $M_0$ in Theorem \ref{M0} has no non-trivial bounded automorphisms. To see that some restriction is necessary, consider the case where $r = 2$ and $n = m = 1$. Then $M_0$ is a graph each of whose connected components consists of an infinite tree with infinite valency, or a single cycle with a collection of such trees attached. Points in the first type of component have $d$-dimension 1, and those in the second type form the $d$-closure of the empty set. It is clear that there are non-trivial automorphisms which stabilise each component (and fix every element in $\cl^d(\emptyset)$), and these are obviously bounded. 

For the rest of this section we assume that  $n,m $ are coprime, if $r = 2$, then $n > m$ and if $r \geq 3$, then $n \geq m$. The proof that there are no bounded automorphisms of $M_0$ uses a technical result, Lemma \ref{48}, whose proof in the case $m > 1$ is surprisingly delicate and makes use of some well-known properties of \textit{Beatty sequences} via the following lemma.

\begin{lemma} \label{4100} Suppose $\ell, b \in \Z$ and $0 < \ell < b$. There is a sequence $(a_i)_{i \in \Z}$ with $a_i \in \{0,1\}$ having the following properties:
\begin{enumerate}
\item $a_{i+b} = a_i$ for all $i \in \Z$;
\item for all $i\in \Z$, $\sum_{i+1 \leq j \leq i+b} a_j = \ell$;
\item for all $i, s\in \Z$ with $s > 0$ we have
\[\frac{1}{s}(-1 +\sum_{i+1 \leq j \leq i+s} a_j ) \leq \frac{\ell}{b}.\]
\end{enumerate}
\end{lemma}

\begin{proof}
Let $\theta = \ell/b$ and note that $0 < \theta < 1$. The Beatty sequence $(\beta_i(\theta))_{i \in \Z}$ is defined as follows. For $ i \in \Z$ let 
\[ \beta_i(\theta) = \floor{i\theta}\]
(where $\floor{x}$ is the largest integer $\leq x$). Let
\[ a_i = \beta_i(\theta) - \beta_{i-1}(\theta).\]
It is easy to see that $a_i \in \{0,1\}$ and $a_{i+b} = a_i$. For part (3) of the Lemma, note that 
\begin{multline*}
\frac{1}{s}(-1 +\sum_{i+1 \leq j \leq i+s} a_j ) = \frac{1}{s}(\beta_{i+s}(\theta) - \beta_i(\theta) -1)\\
= \frac{1}{s}(\floor{(i+s)\theta}- \floor{i\theta} - 1) \leq \frac{i+s}{s}\theta - \frac{\floor{i\theta}+1}{s} \\
< \frac{i+s}{s}\theta - \frac{i\theta}{s} = \theta.
\end{multline*}
A similar calculation shows that
\[ \frac{1}{s}(1 +\sum_{i+1 \leq j \leq i+s} a_j ) > \theta.\]
Thus, for all $i \in \Z$, we have $\frac{1}{s}\sum_{i+1 \leq j \leq i+s} a_j \to \theta$ as $s \to \infty$. The periodicity in (1) then implies (2).\end{proof}

We now state the technical lemma.

\begin{lemma} \label{48}There is $X \subseteq Y \in \C_0$ such that:
\begin{enumerate}
\item $\delta(Y/X) = -1$ and $\vert X \vert \geq 2$;
\item if $U \subseteq Y$ and $X \not\subseteq U$, then $U \cap X \leq U$;
\item if $X \subseteq Z \subset Y$, then $\delta(Z/X) \geq 0$.
\end{enumerate}
\end{lemma}

\begin{proof} Suppose first that $m = 1$. If $r = 2$, take $X = \{ x_0,\ldots, x_n\}$ with no relations on it  and $Y$ is $X$ together with an extra point $y$, where $R(y,x_i)$ holds for all $i$. If $r \geq 3$, do the same, but $X$ also includes an $(r-2)$-tuple $\bar{z}$, and $R(\bar{z}, y, x_i)$ holds. 

So now suppose that $n > m > 1$. We will suppose that $r = 2$: a similar argument to that used above  allows us to deduce the general case. 

Write 
\[ n = ma + c\,\,  \mbox{ with  } 0 < c < m.\]
So $m, c$ are coprime and we can find $\ell, b \in \Z$ with 
\[ \ell m - cb = 1.\]
We can take $0 < b < m$ (take an inverse of $-c$ modulo $m$) and it then follows that $0 < \ell \leq b,c$. If $\ell = b$, then $b=1$.
Note that 
\[ nb - m(ab+\ell) = -1.\]
We now assume that $b \geq 2$ (so in particular, $\ell < b$) and describe the construction of $Y$ (the case $b = 1$ will be considered at the end). 

Let $X$ consist of $(a-1)b + \ell$ points (with no edges). Let $Y = X \cup \{y_0,\ldots, y_{b-1}\}$ with $ab + \ell$ edges as follows:
\begin{enumerate}
\item[(i)] the vertices $y_0,\ldots, y_{b-1}$ form a $b$-cycle (with $R(y_i , y_{i+1})$ holding, where the indices are read modulo $b$);
\item[(ii)] each vertex $y_i$ is adjacent to at least $(a-1)$ of the vertices in $X$;
\item[(iii)] each vertex in $X$ is adjacent to exactly one vertex in $Y\setminus X$.
\end{enumerate}

Thus there are a further $\ell$ edges of $Y$ to be specified. These will be of the form $(x_i, y_i)$ for $i$ in some subset $I \subseteq \{0,\ldots, b-1\}$ of size $\ell$ (and distinct $x_i \in X$). The subset $I$ is chosen so that (3) of the lemma holds. Once we have this, the rest of the lemma follows. Indeed, first note that as $Y$ is a cycle with some extra edges freely amalgamated over its vertices, then $Y \in \C_0$. By construction $\delta(Y/X) = nb - m(ab+\ell) = -1$, so (1) holds. For (2) suppose $\emptyset \neq A \subseteq X$. We claim that $X\setminus A \leq Y \setminus A$, and then (2) follows (by Lemma \ref{411}(2)). To see the claim, note that $\delta((Y\setminus A) /(X\setminus A)) = -1 + m\vert A \vert > 0$, and if $Z \subset Y \setminus X$ then $\delta(Z/(X\setminus A)) \geq \delta(Z/X) \geq 0$, by (3). 

\medskip

To prove (3) (for suitable choice of $I$) it will suffice (by free amalgamation) to show that if $Z \subset Y\setminus X$ is connected, then $\delta(Z/X) \geq 0$. Let $q = \vert \{i \in I : b_i \in Z\}\vert$ and $s = p+q = \vert Z \vert$.

Then 
\[ \delta(Z/X) = sn - m(qa+p(a-1) -p+q-1) = sc - m(q-1).\]

Thus 
\begin{equation}\label{*}\delta(Z/X) \geq 0 \Leftrightarrow \frac{q-1}{s} \leq \frac{c}{m}.\end{equation}

So we need to construct $I$ of size $\ell$ so that for any $s$ consecutive elements of $0,\ldots, b-1$ (read modulo $b$, and with $s < b$), the number of elements $q$ in $I$ satisfies the above inequality. 

\medskip

To the construct $Y$, we let $(a_i)$ be a  sequence as in Lemma \ref{4100} and let:
\[ I = \{ i \in \{0,\ldots, b-1\} : a_i = 1\}.\]

Verifying equation (\ref{*}) amounts to showing that if $0 < s < b$ and $i <b$, then  
$\frac{q-1}{s}  \leq \frac{c}{m}$, 
where $q = \sum_{i+1 \leq j \leq i+s a_j} a_j $. Suppose for a contradiction that $(q-1)/s > c/m$. Recall that $\ell m - cb = 1$, so $\frac{\ell}{b} = \frac{c}{m} + \frac{1}{bm}$. By (3) of Lemma \ref{4100}, $(q-1)/s \leq \ell/b$, so by assumption, we have:
\[ \frac{c}{m} < \frac{q-1}{s} \leq \frac{\ell}{b} = \frac{c}{m}+\frac{1}{bm}.\]
Thus 
\[ 0 < \frac{q-1}{s} - \frac{c}{m} < \frac{1}{bm}.\]
But 
\[ \frac{q-1}{s}- \frac{c}{m} = \frac{(q-1)m - cs}{sm} \geq \frac{1}{sm} > \frac{1}{bm}\]
as $s < b$. This is a contradiction. So $(q-1)/s \leq c/m$ and therefore by equation (\ref{*}), $\delta(Z/X) \geq 0$, as required.

This completes the proof that $Y$ satisfies the properties of Lemma \ref{48}.

\medskip

For the remaining case $b = 1$ (still assuming $r = 2$, without loss of generality) we take $X$ to consist of $a+\ell$ points and $Y\setminus X$ has  a single point $y_0$ which is adjacent to all points in $X$.
\end{proof}

\begin{lemma} \label{49} Suppose $A \in \X$ and $u_0 \in M_0\setminus A$ is basic over $A$. Let $D = \loc(u_0/A)$. Then $\cl^d(A, D) = M_0$.
\end{lemma}

\begin{proof} Suppose $c \in M_0\setminus A$. By Lemma \ref{monod} (3), it will suffice to show that there is a finite tuple $e$ in $D$ with $c \nind^d_A e$. 

Let $A_0 \leq A$ be finite with $d(A_0) = A$. Let $C = \cl_0(c A_0)$. We can assume that $C \cap A = A_0$. Similarly let $B = \cl_0(u_0 A_0)$ and note we can also assume that $B \cap A =  A_0$ (if it is bigger, then replace $A_0$ by the intersection; this will not affect the condition on $C$). 

Let $X \subseteq Y$ be as in Lemma \ref{48} and $k = \vert X \vert$. Note that we can assume that there are no relations on the set $X$. Let $Z$ be the free amalgam of $C$ and $k-1$ copies $B_2,\ldots, B_{k}$ of $B$ over $A_0$. Let $x_1 = c$ and for $i = 2,\ldots, k$ let $x_i \in B_i \setminus A_0$ be the copy of $u_0$ inside $B_i$. Identify the $x_i$ with the points of $X$ and let $E$ consist of the free amalgam  $Z \coprod_X Y$ of $Z$ and $Y$ over $X$.

\smallskip

\textit{Claim:\/} We have $C, B_i \leq E$.

\smallskip

Note that once we have the claim, it follows (as $\emptyset \leq C$) that $E \in \C_0$, so we can assume that $E \leq M_0$. Then $x_2, \ldots, x_k \in D$ and $d(c/A_0, x_2,\ldots, x_k) = d(c/A_0) -1$, so $c \nind^d_A x_2,\ldots, x_k$.

\smallskip

We now prove the claim. By the symmetry of the situation, it is enough to show $C \leq E$. Let $C \subseteq F \subseteq E$. Then $F$ is the free amalgam $F \cap Z \coprod_{F \cap X} F\cap Y$. If $X \not\subseteq F$ then $F\cap X \leq F\cap Y$ (by (2) of Lemma \ref{48}) so $F\cap Z \leq Z$. As $C \leq F\cap Z$ we obtain $C \leq F$. If $X \subseteq F$ and $Y \not\subseteq F$, then similarly (using (3) of Lemma \ref{48}) we have $X = F \cap X \leq F \cap Y$, so again $C \leq F$. 

So now suppose $Y \subseteq F$. Note that $\delta(F\cap Z) \geq d_Z(XC)$ (the dimension in $Z$ of $X \cup C$). So 
\[\delta(F) \geq d_Z(XC) + \delta(Y/X) = d_Z(C) + d_Z(X/C) - 1 \geq \delta(C) + k-2 \geq \delta(C).\]
(Here we have used $C \leq Z$ and (1) of Lemma \ref{48}.)
\end{proof}

\begin{corollary} If $g \in \Aut(M_0/\cl^d(\emptyset))$ is bounded, then there is $E \in \X$ such that $g(\cl^d(E b)) = \cl^d(E b)$ for all $ b \in M_0$.
\end{corollary}

\begin{proof}
This follows from the above and Proposition \ref{unbdd}.
\end{proof}

\begin{remarks} \rm The class $\C_0$ contains some msa extension $X \subset Y$. If we change the structure on $X$ to some other structure in $\C_0$, then then result is still a msa extension in $\C_0$. Furthermore, by `duplicating' the points in $X$ if necessary, we can obtain a msa extension with the property that if $r, r' \in R[Y]$ are distinct and both involve points of $Y\setminus X$ and $X$, then $r \cap r' \cap X = \emptyset$. To do this, replace $X$ by the disjoint union of non-empty $r \cap X$ (for $r \in R[Y]\setminus R[X]$). Then each element of the new $X$ is in exactly one relation in $R[Y] \setminus R[X]$. 
\end{remarks}

\begin{theorem} \label{nobdd} If $g \in \Aut(M_0/\cl^d(\emptyset))$ is bounded, then $g$ is the identity.
\end{theorem}

\begin{proof} Let $E \in \X$ be as in the Corollary: so $g(\cl^d(E b)) = \cl^d(E b)$ for all $b \in M_0$. Let $A \leq E$ be finite and $d(A) = d(E)$. 

\smallskip

\textit{Step 1:\/} If $b \in M_0$ is such that $Ab \leq M_0$ and $\delta(b/A) = n$, then $gb = b$.

\smallskip

\textit{Case 1: $r \geq 3$, $m= n = 1$.\/} Note that $E$ is infinite, so we may take $A$ to be of size at least $r-3$. By using elements of $A$ for the first $r-3$ coordinates in $R$, we can assume without loss that $r = 3$. 

Take $c$ with $c \ind^d_A b$ of the same type as $b$ over $E$. By the boundedness condition on $g$ we have $c, gc \ind^d_A b, gb$. So there are finite $C, B \leq M_0$ with $c, gc \in C$, $b , gb \in B$, $C \cup B \leq M_0$; by enlarging $A$ if necessary we can assume that $E \cap C = A = E \cap B$, and so $C, B$ are freely amalgamated over $A$. 

There is $f \in M_0$ with $R(c,b,f)$ and $CBf \leq M_0$. Note that $d(f/A) = 1$ and $gf \in \cl^d(fA)$,  so there is a finite $A \leq F \leq M_0$ with $\delta(F/A) = 1$ and $f, gf \in F$. Note that $\delta(C/F) = 1$ (otherwise it is zero and then $b \in \cl^d(cA)$). So $\delta(C \cap F/ A) = 0$ and therefore (as $C \cap E = A$) $C \cap F = A$. Similarly $B \cap F = A$. 

Suppose that $\{c, e, b\} \neq \{gc, ge, gb\}$. Then on $C \cup E \cup B$, there are at least 2 extra relations beyond those in the free amalgam over $A$. So 
\[\delta(CEB/A) \leq \delta(C/A)+\delta(E/A)+\delta(B/A) - 2 = 1.\]
But this contradicts $d(cb/A) = 2$. Thus, in particular, $gb = b$.

\medskip

\textit{Case 2: $r \geq 2$, $n > m$.\/} By using elements of $A$ for the first $r-2$ coordinates, we can assume $r = 2$. Let $B = \cl_0(A, gA, b, gb)$ and suppose for a contradiction that $gb \neq b$.

Let $Ab \leq C$ be a simply algebraic extension in $M_0$ with base $U$ containing $b$. We can assume that $b$ is in exactly one relation in $C$. Let $D = C\setminus (Ab)$; so $U \leq U \cup D$ is msa. As $gA \subseteq E$, we can assume that $g(U\cap A) \subseteq A$. We can also assume that $D \cap (B \cup g^{-1}B ) = \emptyset$. Then $gD \cap B = \emptyset$. So both $B \leq B \cup D$ and $B \leq B \cup gD$ are simply algebraic extensions (based on $U$ and $gU = g(U\cap A) gb$ respectively). As $gb \neq b$, we must have $gb \not\in U$, so $D \neq gD$. As the extensions are minimal, it follows that $D \cap gD = \emptyset$. 

Note that $\delta(A) + n = \delta(Ab) = \delta(C) = \delta(AD) +n - m$. So $\delta(AD) = \delta(A)+m$. In particular, $AD \leq C \leq M_0$, so $d(AD) = d(A) +m$. Let $V = \cl_0(A, D, gD)$. We show that $b, gb \not\in V$. Note that $V \subseteq \cl^d(AD)$ (by boundedness of $g$) so $d(V) = d(AD) = d(A)+m$. But $d(Ab) = d(A)+n > d(A) + m$, so $b \not\in V$. As $\cl^d(V)$ is $g$-invariant, we then obtain $gb \not\in V$.

Thus $B \cup V$ has at least 2 more relations in it than in the free amalgam of $B, V$ over $B \cap V$ (a relation from $D$ to $b$ and a relation from $gD$ to $gb$: neither of these is in the free amalgam, by the previous paragraph). So 
\[\delta(BV) \leq \delta(B) + \delta(V) - \delta(B \cap V) - 2m
 \leq \delta(B) + \delta(V) - \delta(A) - 2m.\]
 Now, $\delta(V) =  d(A) +m$. So $\delta(BV) \leq \delta(B) - m$. But this is a contradiction as $m \geq 1$ and $B \leq M_0$. 
 
 \smallskip
 
 \textit{Step 2:\/} If $c \in M_0$ then $gc = c$.
 
 \smallskip
 
 \textit{Case 1:\/} $r \geq 3$, $m = n = 1$. As before, we may assume that $r = 3$. It remains to show that if $c \in E$ then $gc = c$. As $g$ fixes all elements of $\cl^d(\emptyset)$, we may assume $c \not\in \cl^d(\emptyset)$. We may also assume $gc, c \in A$. There exist $e,f \in M_0$ with $Aef \leq M_0$ and $R[Aef] = R[A] \cup \{\{c,e,f\}\}$. Then $Ae, Af \leq Aef$, so by Step 1, $e, f$ are fixed by $g$. It then follows that $c$ is fixed by $g$ (otherwise $\{gc, e, f \} \not\in R$), as required.

\medskip

\textit{Case 2:\/} $r \geq 2$, $n > m$. As before, we may assume that $r = 2$. Let $C = \cl_0(A, c)$. Suppose $s \in \N$. There exist $ b_0 = c, b_1, b_2, \ldots, b_s \in M_0$ such that $R(b_{i-1}, b_i)$ (and no other relations hold on $C \cup \{b_1,\ldots, b_s\}$ outside $C$), and $C b_1\ldots b_s \leq M_0$. It is easy to see that for $t \leq s$ we have $Cb_1\ldots b_t \leq M_0$, $d(b_{t}/Cb_1\ldots b_{t-1}) = n-m$. Moreover,  if $s$ is large enough, then $C b_s \leq M_0$, so $Ab_s \leq M_0$ and $d(b_s/A) = n$. (For this, take $s \geq n/(n-m)$.) It follows from Step 1 that $gb_s = b_s$.

We now show that if $0 \leq t < s$ and  $b_{t+1}$ is fixed by $g$, then so is $b_t$. It follows that $c$ is fixed by $g$, as required. So suppose $b_t$ is not fixed by $g$. Note that $R(b_t, b_{t+1})\wedge R(gb_t , b_{t+1})$. Also, using the boundedness of $g$ we have:
\[n-m = d(b_{t+1}/ Cb_1\ldots b_t) = d(b_{t+1}/C b_1\ldots b_t gb_1\ldots gb_t) \leq d(b_{t+1}/b_t gb_t).\]
In particular, $b_{t+1}\not\in \cl_0(b_t, gb_t)$ and 
\[d(b_{t+1}/b_t gb_t) \leq \delta(b_{t+1}/ \cl_0(b_t, gb_t)) \leq n-2m,\]
because of the edges from $b_{t+1}$ to $b_t, gb_t$.
This is a contradiction (as $m \geq 1$).

\end{proof}

We can now combine the results about the  Hrushovski structure $M_0$ of Theorem \ref{M0} into the following, which is the main result of this section.

\begin{theorem} \label{413} Suppose either that $r  = 2$ and $n > m$, or that $r \geq 3$ and $n \geq m$. Then $\Aut(M_0/\cl^d(\emptyset))$ is a simple group. In fact, if $g \in \Aut(M_0/\cl^d(\emptyset))$ is not the identity then every element of $\Aut(M_0/\cl^d(\emptyset))$ can be written as a product of 96 conjugates of $g^{\pm 1}$

\end{theorem}

\begin{proof} This follows from Corollary \ref{cor39}, Lemma \ref{39}, Lemma \ref{49} and Theorem \ref{nobdd}.
\end{proof}

\begin{remarks}\rm We have been working with symmetric structures in a signature with a single $r$-ary relation.  More generally, suppose we have a signature with relations $R_i$  of arity $r_i$ (for $i \in I$). Suppose $n, m_i$ are positive integers. We define the predimension of a finite structure $A$ to be 
\[\delta(A) = n \vert A \vert - \sum_{i\in I} m_i\vert R_i[A]\vert.\]
Let $\C_0$ consist of such $A$ with $\delta(A') \geq 0 $ for all $A' \subseteq A$. Then we can form the generic structure $M_0$ for $(\C_0, \leq)$ exactly as before. If there is some $i$ such that $m_i \neq 0$ is coprime to $n$, $r_i = 2$ and $n > m_i$, or $r_i \geq 3$ and $n \geq m_i$, then Theorem \ref{413} holds. The argument is the same: for all of the constructions in the proof,  just work with $R_i$ in place of $R$. It should also be clear that our assumption that $R$ is symmetric is not essential. 
\end{remarks}

\section{Further applications}

\subsection{Generalized polygons} For a natural number $n \geq 3$, a \textit{generalized $n$-gon} is a bipartite graph $\Gamma$ of diameter $n$ and girth $2n$. It is \textit{thick} if each vertex has valency at least 3.  In \cite{T},  Hrushovski's amalgamation method from \cite{Hr}  was adapted to produce thick  generalized $n$-gons of finite Morley rank. These are almost strongly minimal and in \cite{GT}, Lascar's result (\cite{L}, Th\'eor\`eme 2) was applied to show that their automorphism groups are simple. This gives new examples of simple groups having a BN-pair which are not algebraic groups.

As with Hrushovski's original construction, an intermediate stage in the construction produces $\omega$-stable generalized $n$-gons $\Gamma_n$ of infinite Morley rank. In this subsection we observe that we can use the results involved in the proof of Theorem \ref{413} in place of Lascar's result to show that these generalized $n$-gons also have simple automorphism group. As in \cite{GT}, $\Aut(\Gamma_n)$ is transitive on ordered $2n$-cycles in $\Gamma_n$, so is also an example of a (non-algebraic) simple group with a spherical BN-pair of rank 2. 

We describe very briefly the construction of $\Gamma_n$ from Section 3 of \cite{T}. Work with a signature which has a unary predicate symbol $P$ and a binary relation symbol $R$ and consider bipartite graphs as structures in this signature, where $P$ picks out the vertices in one part of the partition and $R$ gives the adjacency relation. Vertices in $P$ are called points and those not in $P$ are called lines. Fix a natural number $n \geq 3$.

For a finite (bipartite) graph $A$ define $$\delta(A) = (n-1)\vert A \vert - (n-2) \vert R[A]\vert.$$

As in the previous section, let $\C_0$ consist of the finite bipartite graphs $A$ with $\delta(B) \geq 0 $ for all $B \subseteq A$. If $C \subseteq A$ write $C \leq A$ to mean $\delta(B) \geq \delta(C)$ whenever $C \subseteq B \subseteq A$.

Consider the class $\K_n$ of finite bipartite graphs $A$ which satisfy:
\begin{enumerate}
\item the graph $A$ has no $2m$-cycle, for $m < n$;
\item if $B \subseteq A$ contains a $2m$-cycle for $m > n$, then $\delta(B) \geq 2n+2$.
\end{enumerate}

The following is from (\cite{T}, Corollary 3.13 and Theorem 3.15):

\begin{lemma} We have $\K_n \subseteq \C_0$ and $(\K_n, \leq)$ is an amalgamation class.
\end{lemma}

Let $\Gamma_n$ be the generic structure for the class $(\K_n, \leq)$ (cf.  Theorem \ref{M0}). Then $\Gamma_n$ is a countable generalized $n$-gon which is $\leq$-homogeneous. Lemmas \ref{pe}, \ref{fe} and Corollary \ref{cor1} hold (essentially because of $\leq$-homogeneity and the fact that $\K_n \subseteq\C_0$). As in Corollary \ref{39}, we have:

\begin{corollary} The dimension function $d$ on $\Gamma_n$ satisfies Assumption \ref{dass} and the relation $\ind^d$ is  stationary.
\end{corollary}

\begin{proof} If $X \subseteq Y, Z \in \K_n$ is $d$-closed in $Y, Z$, then the proof of Theorem 3.15 in \cite{T} shows that the free amalgam of $Y$ and $Z$ over $X$ is in $\K_n$. It follows that the class $\X$ of $d$-closures of finite sets in $\Gamma_n$ has the free amalgamation property, and so the proof of Corollary \ref{cor1} gives what we want here.
\end{proof} 

\begin{theorem} The group $\Aut(\Gamma_n)$ is a simple group. In fact, if $1 \neq g \in \Aut(\Gamma_n)$, then every element of $\Aut(\Gamma_n)$ is a product of 96 conjugates of $g^{\pm 1}$.
\end{theorem}

\begin{proof} It follows from (\cite{T}, Corollary 3.13) that $\cl^d(\emptyset) = \emptyset$ for $\Gamma_n$. To prove the theorem, we shall apply Corollary \ref{cor39}. So we first find a suitable basic orbit $D$ and then show that there are no non-trivial bounded automorphisms. The first part is essentially as in the proof of (\cite{T}, Theorem 4.6), but we give a few details.  

If $x \in \Gamma_n$, let $D(x)$ denote the set of vertices adjacent to $x$. Then by the $\leq$-homogeneity, $D(x)$ is a basic orbit over $x$. If $x,y \in \Gamma_n$ are at distance $n$, then there is a bijection definable over $x,y$ from $D(x)$ to $D(y)$ (\cite{T0}, 1.3). Suppose $x_0,\ldots, x_{2n-1}$ is a $2n$-cycle in $\Gamma_n$ with $x_0 \in P$. Then $\Gamma_n$ is in the definable closure of $D(x_0), D(x_1), x_2, \ldots, x_{2n-1}$ (see \cite{T0}, 1.6). If $n$ is odd, there is a vertex $z$ at distance $n$ from both $x_0$ and $x_1$ and therefore $\Gamma_n$ is in the definable closure of $D(x_0), x_1, \ldots, x_{2n-1}, z$. So if we let $A = \{x_0, \ldots, x_{2n-1}, z\}$ and $D = \{ c \in D(x_0) : d(c/A) = 1\}$, then $D$ is a basic orbit over $A$ and $\Gamma_n = \cl^d(A, D)$. 

So now suppose $n$ is even. As in the previous paragraph, it will suffice to show that there is a line $\ell$ and a finite set $A$ with $D(\ell) \subseteq \cl^d(D(x_0), A)$, because $D(x_1)$ is in the definable closure of $D(\ell)$ and some finite set. Let $p_3 \in P$ be at distance $n$ from $x_0$ and let $\ell \not\in P$ be at distance $n-1$ from $x_0, p_3$.  If $k \in D(x_0)$  there is a unique path of length $n-1$ from $k$ to $p_3$. Let $a$ denote the vertex adjacent to $k$ on this path. There is then a unique path of length $n-1$ from $a$ to $\ell$. Let $\phi(k)$ denote the vertex on this path adjacent to $\ell$. So we have a definable map  $\phi : D(x_0) \to D(\ell)$. It can be seen (by considering the paths involved  in this definition of $\phi$) that  that $d(k/x_0, p_3, \ell, \phi(k)) = 0$ for all $k \in D(x_0)$. Thus, if $d(k/x_0, p_3, \ell) = 1$, then $d(\phi(k)/ x_0, p_3, \ell) = 1$.  It follows that the image of $\phi$ contains $D(\ell)\setminus \cl^d(x_0, p_3, \ell)$, so $D(\ell) \subseteq \cl^d(D(x_0), x_0, p_3, \ell)$, as required. 

To show that there are no non-trivial bounded automorphisms, one uses that same proof as in (\cite{GT}, Proposition 6.3), replacing $\acl$ there by $\cl^d$.
\end{proof}

\subsection{$\aleph_0$-categorical structures} \label{Sec52} We recall briefly a variation on the construction method of Section \ref{Sec41} which gives rise to $\aleph_0$-categorical structures. The original version of this is in \cite{Hrpp} where it is used to provide a counterexample to Lachlan's conjecture, and in \cite{Hr97} where it is used to construct a non-modular, supersimple $\aleph_0$-categorical structure. The book \cite{W} (Section 6.2.1) is a convenient reference for this. Generalizations and reworkings of the method (particularly relating to simple theories) can be found in \cite{E:predim}. For the rest of this subsection, assume that $m, n, r, \delta, (\C_0, \leq)$ etc. are as in Section \ref{Sec41}.

In this version of the construction, $d$-closure is uniformly locally finite. Suppose  $f : \R^{\geq 0} \to \R^{\geq 0}$ is a continuous, increasing function with $f(x) \to \infty$ as $x \to \infty$. Let 
$$\C_f = \{A \in \C_0 : \delta(X) \geq f(\vert X \vert)\,\, \forall X \subseteq A\}.$$ 
Note that if $X \subseteq A \in \C_f$ then 
$$\delta(X) \geq \delta(\cl^d(X)) \geq f(\vert \cl^d(X)\vert)$$
so $\vert \cl^d_A(X) \vert \leq f^{-1}(\delta(X)) \leq f^{-1}(n\vert X \vert).$

If $B \subseteq A \in \C_f$ and  $\cl^d_A(B) = B$, then we write $B \leq_d A$. For suitable choice of $f$ (call these \textit{good} $f$), $(\C_f, \leq_d)$ has the free $\le_d$-amalgamation property: if $A_0 \leq_d A_1, A_2 \in \C_f$ then $A_i \leq_d A_1\coprod_{A_0} A_1 \in \C_f$. In this case we have an associated countable \textit{generic structure} $M_f$. So $M_f$ is $\leq_d$-homogeneous and the set $\X$ of finite $d$-closed subsets of $M_f$ is (up to isomorphism) $\C_f$. As $d$-closure is uniformly locally finite, the structure $M_f$ is $\aleph_0$-categorical (by the Ryll - Nardzewski Theorem). Algebraic closure in $M_f$ is equal to $d$-closure.

\begin{remarks}\label{frem}\rm
To construct good functions, we can take $f$ which are piecewise smooth and where the right derivative $f'$ satisfies $f'(x) \leq 1/x$ and is non-increasing, for $x \geq 1$. The latter condition implies that $f(x+y) \leq f(x) + yf'(x)$ (for $y \geq 0$). It can be shown that under these conditions, $\C_f$ has the free $\leq_d$-amalgamation property. Also note that if $f'(x) \leq 1/x$ for all $x \geq x_0$, then for $y \geq x \geq x_0$ we have $f(y) \leq f(x) + \log(y-1) - \log(x-1)$.
\end{remarks}

\begin{assumption} \label{ass} \rm Henceforth, we assume that if $r = 2$, then $n > m$ and if $r\geq 3$, then $n \geq m$. We suppose that $f$ is a good function.  We will assume that $f(0) = 0$ and $f(1) > 0$, therefore $\cl^d(\emptyset) = \emptyset$. We shall also assume that $f(1) = n$. Thus if $X \in \C_f$ and $\vert X \vert \geq 2$, then $\delta(X) \geq f(\vert X \vert) > n$. In particular $\{x \} \leq_d X$ for all $ x \in X$.
\end{assumption} 

Let $G = \Aut(M_f)$.

As before, we write $\ind^d$ for $d$-independence in $M_f$. This is not stationary. If $A \leq_d  C \in \X$ and $b_0 \in M_f$, then $\{ b \in \loc(b_0/A) : b \ind^d_A C\}$ need not be a single $G_C$-orbit: the orbits are determined by the $d$-closures $\cl^d(bC)$. Clearly $\cl^d(bC) \supseteq \cl^d(bA) \cup C$ and as in Lemma \ref{421} it can be shown that $\cl^d(bA) \cap C = A$, $\cl^d(bA), C$ are freely amalgamated over $A$ and $\cl^d(bA) \cup C \leq M_f$ if and only if $b\ind^d_A C$. The closure operation $\cl^d$ on $M_f$ is finitary, invariant and subsumes definable closure.

\begin{definition} \rm Suppose $A \leq_d C \in \X$ and $b$ is a tuple of elements of $M_f$. Write $b \perp_A C$ to mean that $b \ind^d_A C$ and $\cl^d(bC) = \cl^d(bA)\cup C$. Note that in this case, $\cl^d(bC)$ is the free amalgam of $\cl^d(bA)$ and $C$ over $A$.
\end{definition}

The following is straightforward:

\begin{lemma} \label{perpsi} The relation $\perp$ is a stationary independence relation on $M_f$ compatible with $\cl^d$. \hfill$\Box$
\end{lemma}

We will use Theorem \ref{thmTZ} to show that, under some restrictions, the group  $G = \Aut(M_f)$ is simple. The proof is similar to that in the previous sections, but we need to make some modifications as the dimension function does not give rise to a stationary independence relation. 

Suppose $A \in \X$ and $b \in M_f$. We shall continue to say that $b$ is basic over $A$ if $b \not\in A$ and whenever $A \leq_d C \in \X$ and $d(b/C) < d(b/A)$, then $b \in C$. Recall also that $M_f$ is monodimensional if for all basic orbits $D = \loc(b/A)$ (for $A \in \X$) there is $B \in \X$ with $A \subseteq B$ and $M_f = \cl^d(B, D)$. In fact, in the examples below where we verify this, we will take $B = A$.

As before, we say that $g \in G$ is $d$-bounded over $A \in \X$ if there is $A \subseteq C \in \X$ and $b \in M_f$ which is basic over $C$ such that for all $b' \in \loc(b/C)$ we have $gb' \in \cl^d(b'C)$.

\begin{lemma} \label{notbdd2}Suppose $M_f$ is monodimensional and $g \in \Aut(M_f)$ is $d$-bounded (over some element of $\X$). Then $g = 1$.
\end{lemma}

\begin{proof} By  Proposition \ref{unbdd} there is  $E\in \X$ such that $g$ stabilizes every $B \in \X$ containing $E$. In particular, $g$ fixes all $b \in M_f\setminus E$ for which $Eb \leq_d M_f$. 

Let $c,c'$ be distinct elements of $M_f$ and $C = \cl^d(E, c, c')$. First suppose that $r > 2$. Consider the structure $B$ consisting of $c$ together with $r-1$ points $b_1, \ldots, b_{r-1}$ such that $R[B]$ is the single relation $\{c, b_1, \ldots, b_{r-1}\}$. Then $B \in \C_f$ and $c \leq_d B$. By Assumption \ref{ass}, the free amalgam $U$ of $C$ and $B$ over $c$ is in $\C_f$, so we may suppose $U \leq_d M_f$. One calculates that $Eb_i \leq_d U$ for each $i$ (this uses that $r >2$), therefore the $b_i$ are fixed by $g$. As $g$ stabilizes $E, C$ and $U$, it is then clear that $gc \neq c'$. But this holds for all $c'\neq c$, so in fact, $gc = c$.

Now suppose that $r = 2$ (and $n > m$). Take $b \perp C$. Suppose $c, e_1, \ldots, e_s, b$ is a simple path with endpoints $c, b$. If $s > m/(n-m)$ then $cb \leq_d ce_1\ldots e_s b$. As $cb \leq_d Cb$ we may use free amalgamation over $cb$ to find such a path with $U = Ce_1\ldots e_s b \leq_d M_f$. Then $gb = b$ and $g$ stabilizes $E, C, U$. There is a path from $b$ to $c$ whose internal vertices are in $U\setminus C$, but there is no such path to $c'$. So $gc \neq c'$, and it follows that $gc = c$.
\end{proof}

\begin{proposition} \label{59}Suppose $M_f$ is monodimensional, $A \in \X$ and $D$ is a basic orbit over $A$. Suppose $1 \neq g \in \Aut(M_f/A)$. 
\begin{enumerate}
\item If $c \in M_f$ and $A \subseteq B \in \X$, then there is $c' \in \loc(c/B)$ with $gc'\ind^d_B c'$.
\item There is $\h \in G_A$ such that the commutator $\g = [g, \h]$ moves almost maximally over $A$ with respect to $\perp$, that is, if $a'\in M_f$ and $A \subseteq X \in \X$,  there is $a \in \loc(a'/X)$ such that $\g a \perp_X a$.
\end{enumerate}
\end{proposition}

\begin{proof} (1) This follows from Lemma \ref{notbdd2} and Proposition \ref{thm1}.

(2) We build $\h$ by a back-and-forth construction as in the first part of the proof of Theorem \ref{thm2}. During the `forth' step we shall ensure that $\g$ moves almost maximally with respect to $\perp$ (over $A$). So suppose we have constructed a partial automorphism $h : U \to V$ (fixing $A$) and $X$, $a'$ are given. By extending $h$ arbitrarily, we may assume that $U \supseteq X, gX, h^{-1}ghX$.

\smallskip

\textit{Claim 1:\/} We can choose $a \in \loc(a'/X)$ such that $a \perp_X U, g^{-1}U$ and $ga \ind^d_U a$.

To do this, take $a'' \in \loc(a'/X)$ with $a'' \perp_X U, g^{-1}U$ (by Extension). Then by (1), there is $a\in \loc(a''/\cl^d(U, g^{-1}U))$ with $ga \ind^d_{U, g^{-1}U} a$. It follows from Transitivity (for $\ind^d$) that $ga \ind^d_{U} a$, as required.

Similarly, we can take $b \in h\loc(a'/U)$ with $b \perp_{hX} {V, g^{-1}V}$ and $gb \ind^d_V b$. Extend $h$ by setting $ha = b$.

Note that $h^{-1}\loc(gb/\cl^d(V, b))$ is an orbit over $\cl^d(U, a)$. We choose $e$ in this with $e \perp_{U, a} ga$ and extend $h$ further by setting $he = gb$.

\smallskip

We have that $\cl^d(e, U, a) \perp_{U,a} \cl^d(ga, U, a)$. Intersecting this $d$-closed free amalgam with $Y = \cl^d(U, e, ga)$ we obtain another $d$-closed free amalgam, so $e \perp_Z ga$, where $Z = \cl^d(U, a) \cap Y$. 

\smallskip

\textit{Claim 2:\/} We have $Z = U$, so  $e \perp_U ga$.

By Claim 1 we have $d(ga, a/U) = d(ga/U) + d(a/U)$, and similarly $d(gb/V, b) = d(gb/V)$. So we have:
\[d(e/U,a,ga) = d(e/U,a) = d(gb/V, b) = d(gb/V) = d(e/U),\]
where the second and fourth of these come from applying $h$. It then follows that $a, ga, U$ are $d$-independent over $U$, so $a \ind^d_U ga , e$. In particular, $\cl^d(u, a) \cap \cl^d(U, ga, e) = U$.

\smallskip

\textit{Claim 3:\/} We have $e\perp_{gX} ga$. 

By Claim 1, $U \perp_{gX} ga$ so $\cl^d(U, ga) = U \coprod_{gX} E_2$, where $E_2 = \cl^d(gX, ga)$. 

By choice of $b$ we have $gb \perp_{ghX} gV, V$, so (applying $h^{-1}$) $e \perp_{h^{-1}ghX} U$. Thus $\cl^d(U, e) = U \coprod_{h^{-1}ghX} E_1$ where $E_1 = \cl^d(h^{-1}ghX, e)$. 

Let $A_i = E_i \cap U$. So $A_1 = h^{-1}ghX$ and $A_2 = gX$. Let $W = \cl^d(A_1, A_2)$. By Claim 2, $U \cup E_1 \cup E_2 \leq_d M_f$. We also have $W \cup E_1 \cup E_2 \leq_d U \cup E_1 \cup E_2$, so $E_1 \perp_W E_2$, that is:
\[ E_1 \perp_{A_1, A_2} E_2.\]
As $a \perp_X g^{-1}U$, we have (applying $g$) $E_2 \perp_{A_2} U$. So $E_2 \perp_{A_2} E_1$. By Transitivity we obtain $E_1 \perp_{A_2} E_2$, which gives the claim. 

By applying $g^{-1}$ to  Claim 3 we obtain:
\[ [g,h]a \perp_X a\]
which is what we wanted to do in this step of the construction.
\end{proof}

The following is the main result of this section (we are still assuming \ref{ass} here).

\begin{theorem} \label{511} Suppose $M_f$ is monodimensional and $1 \neq g \in \Aut(M_f)$. Then every element of $\Aut(M_f)$ is a product of 192 conjugates of $g^{\pm 1}$. In particular, $\Aut(M_f)$ is a simple group.
\end{theorem}

\begin{proof} Let $G = \Aut(M_f)$. Note that $\cl^d(\emptyset) = \emptyset$. Let $A \in \X$ be such that there is a basic orbit $D$ over $A$. It is easy to show that there is a non-identity commutator $g_1$ of $g$ which fixes every element of $A$. By Proposition \ref{59}, by taking a further  commutator with an element of $G_A$ we obtain some $g_2 \in G_A$ which moves almost maximally over $A$ (with respect to $\perp$). It follows from Theorem \ref{thmTZ} that every element of $G_A$ is a product of 16 conjugates of $g_2$. As $g_2$ is a product of 4 conjugates of $g^{\pm 1}$, it follows that every element of $G_A$ is a product of 64 conjugates of $g^{\pm 1}$. As in the final part of the proof of Theorem \ref{thm2}, $G$ is the product of three conjugates of $G_A$: hence the result.
\end{proof}

\medskip

We believe that under the conditions of Assumption \ref{ass}, the structure $M_f$ should be monodimensional. However, proving this  appears to require an extremely technical argument and we   only have  a full proof in some special cases.

\begin{example}\label{e511} \rm
Suppose that $r \geq 3$ and $m = n = 1$; so $\delta(A) = \vert A \vert - \vert R[A]\vert$. Suppose $f$ is as in Remarks \ref{frem} and also that Assumption \ref{ass} holds. 

If $A \in \X$ and $b \in M_f \setminus A$ then $d(b/A) = 1$ so $b$ is basic over $A$. Let $D = \loc(b/A)$. We show that $M_f = \cl^d(A, D)$. 

\medskip

\textit{Step 1.\/} There is $c \in \cl^d(A, D)$ with $c \perp A$.

Let $B = \cl^d(A, b)$ and let $F$ be the free amalgam of copies $B_1, \ldots, B_{r-1}$ of $B$ over $A$, with $b_i \in B_i$ being the copy of $b$ inside $B_i$. Let $E = F \cup \{c\}$ where $R(b_1,\ldots, b_{r-1}, c)$ holds and this is the only relation in $E$ involving $c$. We show that:
\begin{enumerate}
\item[(i)] $E \in \C_f$;
\item[(ii)] $B_i \leq_d E$;
\item[(iii)] $Ac \leq_d E$.
\end{enumerate}
Note that once we have this, it follows that we may assume $E\leq_d M_f$ and so (by (ii))  $b_1,\ldots, b_{r-1} \in D$. Moreover, $c \in \cl^d(A, b_1,\ldots, b_{r-1})$ and (by (iii)) $A \perp c$, which finishes Step 1.

For (i), note of course that $F \in \C_f$. Let $Y \subseteq E$. We want to show that $\delta(Y) \geq f(\vert Y\vert)$. We may assume that $c,b_1,\ldots, b_{r-1} \in Y$ and $Y \leq_d E$. In the following, if $C \subseteq E$, let $Y_C = Y \cap C$. 

If $Y_A = \emptyset$ then $Y$ is obtained by free amalgamation over the $b_i$ from $\{b_1,\ldots, b_{r-1}, c\}$ and the $Y_{B_i}$, so is in $\C_f$. So we may assume that $Y_A \neq \emptyset$. Also, if $\vert Y_{B_i} \setminus A\vert = 1$ for all $i$, then as $d(b_i / A) = 1$, there are no relations between $Y_A$ and $\{b_1,\ldots, b_{r-1}, c\}$ and $Y$ is again a free amalgam. So we may also assume that $2 \leq \vert Y_{B_1}\setminus A\vert \geq \vert Y_{B_i}\setminus A \vert$. In particular, $\vert B_1 \vert \geq 3$.

Now we compute that 
\[ \delta(Y) = \delta(Y_F) = \delta(Y_{B_1}) + \sum_{i \geq 2} \delta(Y_{B_i}/Y_{B_1}) \leq \delta(Y_{B_1}) + (r-2).\]
Also
\[ \vert Y \vert = 1 + \vert Y_{B_1}\vert + \sum_{i \geq 2} \vert Y_{B_i} \setminus A\vert \leq 1+ \vert Y_{B_1}\vert + (r-2)\vert Y_{B_1}\setminus A\vert.\]

As in Remarks \ref{frem}
\[f(\vert Y\vert) \leq f(\vert Y_{B_1}\vert) + \log\left(\frac{\vert Y_{B_1}\vert + (r-2)\vert Y_{B_1}\setminus A\vert}{\vert Y_{B_1}\vert - 1}\right).\]

So to prove that $\delta(Y) \geq f(\vert Y\vert)$ it will suffice to show that 
\[ r-2 \geq \log\left(\frac{\vert Y_{B_1}\vert + (r-2)\vert Y_{B_1}\setminus A\vert}{\vert Y_{B_1}\vert - 1}\right).\]

As $\vert Y_A \vert \geq 1$ and $\vert Y_{B_1}\setminus Y_A\vert \geq 2$ we have:

\[ \frac{\vert Y_{B_1}\vert + (r-2)\vert Y_{B_1}\setminus A\vert}{\vert Y_{B_1}\vert - 1} \leq (r-1)+ \frac{1}{2},\]
and the required inequality holds as $r \geq 3$. This completes the proof of (i).

We now verify (ii); without loss we take $i = 1$. Suppose $B_1 \subset Y \subseteq   E$. We need to show that $\delta(B_1) < \delta(Y)$. We may assume that $Y \leq_d E$ and also that $b_1,\ldots, b_{r-1}, c \in Y$ (otherwise what we want follows from free amalgamation). But then $Y = E$ and $\delta(E) = \delta(B_1) + (r-2) > \delta(Y)$. 

For (iii), suppose $Ac \subset Y \subseteq E$. If $Y$ does not contain all of $b_1,\ldots, b_{r-1}$, then $\delta(Y) = \delta(Y_F) +1 > \delta(A) + 1 = \delta(Ac)$. On the other hand, if $Y$ contains all of $b_1,\ldots, b_{r-1}$, then $\delta(Y) \geq \delta(A) + (r-1) > \delta(Ac)$. This completes Step 1.

\medskip

From Step 1 and Stationarity, it follows that $\cl^d(A, D) \supseteq \{ e \in M_f : e \perp A\}$. So to show that $\cl^d(A, D) = M_f$ it will suffice to show:

\medskip

\textit{Step 2.\/} If $a \in M_f \setminus A$, there exist $e_1,\ldots, e_{r-1} \in M_f$ with $e_i \perp A$ and $a\in \cl^d(A, e_1,\ldots, e_{r-1})$. 

To see this, let $C = \cl^d(A, a)$ and let $F$ be the free amalgam of this over $a$ with the structure on points $\{a,e_1,\ldots, e_{r-1}\}$ which has a single relation $R(a,e_1,\ldots, e_{r-1})$. As $A \leq_d F$, we can assume that $F \leq_d M_f$. Moreover, an easy calculation shows that $Ae_i \leq_d F$ and so $e_i \perp A$ for all $i$. But $a \in \cl^d(e_1,\ldots, e_{r-1})$ so we have completed Step 2.

\end{example}

\begin{example}\label{e512} \rm
Suppose as in \cite{Hrpp} that $r = 2$, $n = 2$ and $m=1$. So we are considering graphs $A$ and $\delta(A) = 2\vert A \vert - e(A)$ where $e(A)$ denotes the number of edges in $A$.  We take $f(0) = 0$, $f(1) = 2$, $f(2) = 3$ and $f'(x) \leq 1/x$ non-increasing for $x \geq 2$ as in Remarks \ref{frem}. So if $A \in \C_f$, then vertices and edges are $d$-closed in $A$. Moreover $f(x) \leq 3+ \log(x-1)$ for $x \geq 2$; more generally, $f(y) \leq f(x) + \log(y-1) - \log(x-1)$ for $2 \leq x \leq y$.

By free amalgamation, $\C_f$ contains paths $P_\ell$ of arbitrary length $\ell$. One easily computes that if $u, v$ are the endpoints of $P_\ell$ then $uv \leq_d P_\ell$ iff $\ell \geq 3$. In particular (using free amalgamation), $\C_f$ contains a $6$-cycle, but need not contain shorter cycles.

The strategy for verifying monodimensionality is as in the previous example, but the details are considerably more complicated. Suppose $A \in \X$ and $\loc(b/A)$ is any $G_A$-orbit on $M_f \setminus A$. We shall show that there exist $b_0 , \ldots, b_{s-1} \in \loc(b/A)$ and $c \in \cl^d(b_0, \ldots, b_{s-1}, A)$ such that $c \perp A$. So $\cl^d(A, \loc(b/A))$ contains $\{e : e\perp A\}$. We then observe that $\cl^d(A, \{e : e\perp A\}) = M_f$.

In order to do this, we construct various graphs and verify that they are in $\C_f$. 

\medskip

\textit{Step 1.\/} Let $s \in \N$ be sufficiently large. Construct a graph with vertices $C = \{c_0,\ldots, c_{s-1}\}$ and $D = \{d_0,\ldots, d_{s-1}\}$ such that:
\begin{itemize}
\item $c_0,d_0, c_1, d_1, \ldots, c_{s-1}, d_{s-1}$ is a $2s$-cycle;
\item the remaining edges on $CD$ form a single $s$-cycle on $D$ and $CD$ has girth at least $6$.
\end{itemize}

To do this, we can take adjacencies in $D$ to be $d_i \sim d_{i+\ell}$ where the indices are read modulo $s$ and $\ell$ is chosen coprime to $s$ and $6 \leq \ell < s/12$. 

\textit{Step 2.\/} We have $CD \in \C_f$.

Note that  as $s$ is large, $\delta(CD) = s > 3 + \log(2s-1) \geq f(2s) = f(\vert CD\vert)$. Let $X \subset CD$. We need to show that $\delta(X) \geq f(\vert X \vert)$. We may assume that $X \leq_d CD$. Write $X_D = D \cap X$ and use similar notation throughout what follows. We have $X_D \subset D$, so 
\[ \delta(X_D) \geq 2 \vert X_D \vert - (\vert X_D \vert -1) = \vert X_D \vert + 1.\]
Consider the valencies of vertices in $X_C$ within $X$. There are at most $\vert X_D\vert -1$ of valency 2 and those of valency at most 1 contribute at least 1 to $\delta(X /X_D)$. Thus
\[ \vert X_C \vert \leq \delta(X/X_D) + \vert X_D\vert -1,\]
so 
\[ \delta(X) \geq \vert X_C\vert - \vert X_D\vert + 1 + \delta(X_D) \geq \vert X_C\vert + 2.\]
Also,
\[\delta(X) = 2\vert X_C\vert + 2\vert X_D\vert -e(X_C , X_D) - e(X_D) \geq \delta(X_D)\]
as $e(X_C, X_D)$, the number of edges between $X_C$ and $X_D$, is at most $2\vert X_C\vert$. So
\[ \delta(X) \geq \delta(X_D) \geq \vert X_D \vert +1.\]
We therefore obtain:
\[\delta(X) \geq \frac{1}{2}(\vert X \vert +3).\]
As $f(x) \leq 3 +2 \log(x-1)$, we have $\delta(X) \geq f(\vert X \vert)$ if $\vert X\vert \geq 7$. If $\vert X \vert  \leq 6$ then $X$ is either a $6$-cycle or has no cycles, so is in $\C_f$.

\smallskip

\textit{Step 3.\/} If $X \leq_d CD$ and $X$ is the $d$-closure in $CD$ of $X_C$, then $\vert X\vert \leq 4\vert X_C\vert - 3$.

\smallskip

This follows from the fact that $0 \geq \delta(X/X_C) \geq \frac{1}{2}(\vert X\vert +3) - 2 \vert X_C\vert$.

\smallskip

\textit{Step 4.\/} Let $B$ consist of copies $B_0, \ldots, B_{s-1}$ of $B' = \cl^d(A, b)$ freely amalgamated over $A$, with $b_i$ the copy of of $b$ inside $B_i$. Let $E = B \cup C \cup D$ with edges as in $B$, $C \cup D$ and additional edges $b_i \sim c_i$ for $i = 0,\ldots, s-1$. Note that $\delta(E) = \delta(A) + s\delta(B'/A) = \delta(B)$ and $\vert E \vert  = \vert A \vert + s\vert B'\setminus A\vert + 2s = \vert A \vert + s(\vert B'\setminus A\vert + 2)$. For sufficiently large $s$ we have $\delta(E) \geq f(\vert E \vert)$ (by the logarithmic growth of $f$). 

Suppose $Y \subset E$; we claim that $\delta(Y) \geq f(\vert Y \vert)$, so $E \in \C_f$. We may assume that $Y \leq_d E$.  It is clear that $E$ is the free amalgam of $BC$ and $CD$ over $C$ and it is easy to check that $C \leq_d BC$. So $Y_C \leq_d Y_{BC}$. 

Let $Y_C'$ be the $d$-closure of $Y_C$ inside $CD$. So $Y_C' \subseteq Y_{CD}$ and $Y_C' \cap C = Y_C$. Then $Y_B \cup Y_C'$ is a free amalgam over $Y_C$ and $Y_C' \leq_d Y_B \cup Y_C'$. Moreover, $Y_C' \leq Y_{CD}$; so it will suffice to show that $Y_B \cup Y_C' \in \C_f$. Thus we may assume $Y_C' = Y_{CD}$. In particular, by Step 3, we may assume that $\vert Y_{CD}\vert \leq 4t - 3$, where $t = \vert Y_C \vert$. We can assume $t \geq 2$.

We may assume that $\delta(Y_{B_i}/Y_A) \leq 1$ for all $i$. Then we may further assume that  $b_i \in Y$ iff $c_i \in Y$. (If $c_i \in Y$ and $b_i \not\in Y$, then adding $b_i$ into $Y$ increases the size of $Y$ without increasing $\delta$; conversely if $b_i \in Y$ but $c_i$ is not, then $Y_{B_i}$ is freely amalgamated with the rest of $Y$ over $Y_A$.) Similarly we can assume that if $Y_{B_i} \supset Y_A$ then $b_i \in Y_i$. It follows that $\delta(Y_B/Y_A) = t$. 

Choose $i$ such that $\vert Y_{B_i} \setminus Y_A\vert $ is as large as possible; say $i = 1$ and the size is $k$. Then 
\[ \vert Y\vert  = \vert Y_B\vert + \vert Y_{CD}\vert \leq \vert Y_{B_1}\vert + (t-1)k + 4t-3.\]
Also 
\[ \delta(Y) = \delta(Y_B) + \delta(Y_{CD}) - e(Y_B, Y_C) \geq (\delta(Y_{B_1}) + (t-1)) + (t+2) - t \]
using the inequality $\delta(Y_{CD}) \geq t+2$ from Step 2, and so:
\[\delta(Y) \geq \delta(Y_{B_1}) + t + 1.\]

So it will suffice to show that 
\[\delta(Y_{B_1}) + t + 1 \geq f(\vert Y_{B_1}\vert + (t-1)k + 4t-3).\]
By the logarithmic nature of $f$, and $\delta(B_1) \geq f(\vert B_1\vert)$, this will follow from:
\[t+1 \geq \log((t-1)(k+4)) - \log(\vert Y_{B_1}\vert - 1).\]
It is easily checked that this is the case (as $t \geq 2$ and $\vert Y_{B_1} \vert \geq k+1)$. This finishes the proof that $E \in \C_f$.

\smallskip 

\textit{Step 5.\/} If $e \in D$, then $Ae \leq_d E$. To see this, let $Ae \subset X \subseteq E$. As $E$ is a free amalgam over $C$
\[\delta(X) = \delta(X_{BC}/X_C) + \delta(X_{CD}).\] 
It is straightforward to see that this is greater than $\delta(Ae) = \delta(A)+2$.

\smallskip

\textit{Step 6.\/} We have $B_i \leq_d E$. This follows from the the calculations in Step 4.

\smallskip

It follows that $A \leq_d E$, so we may assume that $E \leq_d M_f$. As $\delta(E) = \delta(B)$, we have $E = \cl^d(B)$. By Step 6, each $b_i$ is in  $\loc(b/A)$. By Step 5, we have that $A\perp e$ for $e \in D$. It follows that $\cl^d(A,\loc(b/A))$ contains $\{ e \in M_f: e \perp A\}$. 

\medskip

To conclude, we show that $\cl^d(A, \{e: e\perp A\}) = M_f$. Let $x \in M_f \setminus A$ and $X = \cl^d(x, A)$. Using the above construction we can find $V \in \C_f$ and distinct $b_1, \ldots, b_s, y \in V$ such that $y \in \cl^d(b_1,\ldots, b_s)$ and $y$ is not adjacent to any of the $b_i$. The latter implies that $yb_i \leq V$. Identify $y$ with $x$ and form the free amalgam $U$ of $V$ and $X$ over $x$. This is in $\C_f$ so we may assume $U \leq_d M_f$. Using that $xb_i \leq V$, it is straightforward to check that $b_i \perp A$, and so $x \in \cl^d(A, \{e: e\perp A\})$, as required. It follows that $M_f$ is monodimensional.

\end{example}

\subsection{Concluding remarks}

Hrushovski's paper \cite{Hrpp} uses a further variation on the construction method of the previous subsection to produce stable, $\aleph_0$-categorical structures which are not one-based. In this variation of the construction, the predimension is given by 
\[\delta(A) = \vert A \vert - \alpha\vert R[A]\vert\]
where $\alpha \in \R^{\geq 0}$ is \textit{irrational}. For certain $\alpha$ one defines a control function $f_\alpha : \R^{\geq 0} \to \R^{\geq 0}$ such that $\C_{f_\alpha}$ is a free amalgamation class and the Fra\"{\i}ss\'e limit $M_\alpha$ is stable and $\aleph_0$-categorical. The details of this can be found in (\cite{W0}, Example 5.3). Forking independence gives a stationary independence relation on $M_\alpha$ and it would be interesting to investigate simplicity (or otherwise) of $\Aut(M_\alpha)$ using Theorem \ref{thmTZ}.

In his paper \cite{L}, Lascar also proves a \textit{small index property} for countable, saturated almost strongly minimal structures and it would be interesting to know whether these methods can be used to prove that such a property also holds for the structures $M_0$ and $M_f$ (for good $f$) of Sections \ref{Sec41} and \ref{Sec52}. More specifically, we ask:
\begin{itemize}
\item Suppose $G$ is $\Aut(M_0)$ or $\Aut(M_f)$ and $H \leq G$ is of index less than $2^{\aleph_0}$ in $G$. Does there exist $A \in \X$ such that $H \geq G_A$?
\end{itemize}

In the case where $G = \Aut(M_0)$, it seems likely that Lascar's methods work, though we have not checked all of the details. For the case where $G = \Aut(M_f)$, the following problem is relevant:

\begin{itemize}
\item Suppose $A_i, B_i \leq_d M_f$ are finite and $h_i : A_i \to B_i$ is an isomorphism  (for $i = 1,\ldots, n$). Do there exist $D \in \X$ with $A_i , B_i \leq_d D$ and $g_i \in \Aut(D)$ such that $g_i \supseteq h_i$ for all $i \leq n$?
\end{itemize}

\end{document}